\documentclass[onefignum,onetabnum,final]{siamart220329}

\usepackage[utf8]{inputenc}
\usepackage[english]{babel}
\usepackage{amsmath}
\usepackage{color}
\usepackage{amsfonts}
\usepackage{amssymb}
\usepackage{hyperref} 
\usepackage{graphicx}
\usepackage{algorithmic}
\usepackage{wrapfig}
\usepackage{caption}

\newcommand{\ie}{{\em i.e.,~}}

\newcommand{\eg}{{\em e.g.,~}}

\newcommand{\RR}{\mathbb{R}}
\newcommand{\NN}{\mathbb{N}}

\newcommand{\diag}{\mathrm{diag}}

\newcommand{\eqdef}{\ensuremath{\stackrel{\mbox{\upshape\tiny def.}}{=}}}
\newcommand{\ofrac}[2]{\frac{#1}{#2}}

\newcommand{\sinverse}{\sharp}
\newsiamremark{remark}{Remark}

\newtheorem{assumption}[theorem]{Assumption}

\newcommand{\cent}{L_\mathrm{center}}
\newcommand{\var}{\mathrm{var}}
\newcommand{\dH}{d_\mathcal{H}}
\newcommand{\soll}{\hat P(\theta)}
\newcommand{\sol}{\hat P}
\newcommand{\cost}{\mathcal{L}}
\newcommand{\const}{\mathrm{const}}
\newcommand{\Rpos}{\mathbb{R}_{>0}}
\newcommand{\Rzpos}{\mathbb{R}_{\geq0}}
\DeclareMathOperator{\Ent}{Ent}

\definecolor{linkcolor}{RGB}{83,83,182}
\definecolor{citecolor}{RGB}{128,0,128}
\hypersetup{
    colorlinks=true,
    citecolor=citecolor,
    linkcolor=linkcolor,
    urlcolor=linkcolor
}

\ifpdf
\hypersetup{
  pdftitle={The derivatives of Sinkhorn--Knopp converge},
  pdfauthor={Edouard Pauwels and Samuel Vaiter}
}
\fi

\headers{The derivatives of Sinkhorn--Knopp converge}{Edouard Pauwels and Samuel Vaiter}

\title{The derivatives of Sinkhorn--Knopp converge\thanks{Submitted \today.\funding{E. P. acknowledges the financial support of the AI Interdisciplinary Institute ANITI funding under the grant agreement ANR-19-PI3A-0004, Air Force Office of Scientific Research, Air Force Material Command, USAF, under grant numbers FA9550-19-1-7026, and ANR MaSDOL 19-CE23-0017-01. S. V. acknowledges the support ANR GraVa, grant ANR-18-CE40-0005.}}}
\author{Edouard Pauwels\thanks{IRIT, CNRS, Universit\'e de Toulouse, ANITI, Toulouse, France. Institut universitaire de France (IUF). (\email{edouard.pauwels@irit.fr}).} \and Samuel Vaiter\thanks{CNRS \& Université Côte d’Azur, CNRS, LJAD, France (\email{samuel.vaiter@cnrs.fr}).}}

\begin{document}

\maketitle

\begin{abstract}
    We show that the derivatives of the Sinkhorn--Knopp algorithm, or iterative proportional fitting procedure, converge towards the derivatives of the entropic regularization of the optimal transport problem with a locally uniform linear convergence rate.
\end{abstract}



\begin{wrapfigure}{r}{5.5cm}
    \includegraphics[width=0.9\linewidth]{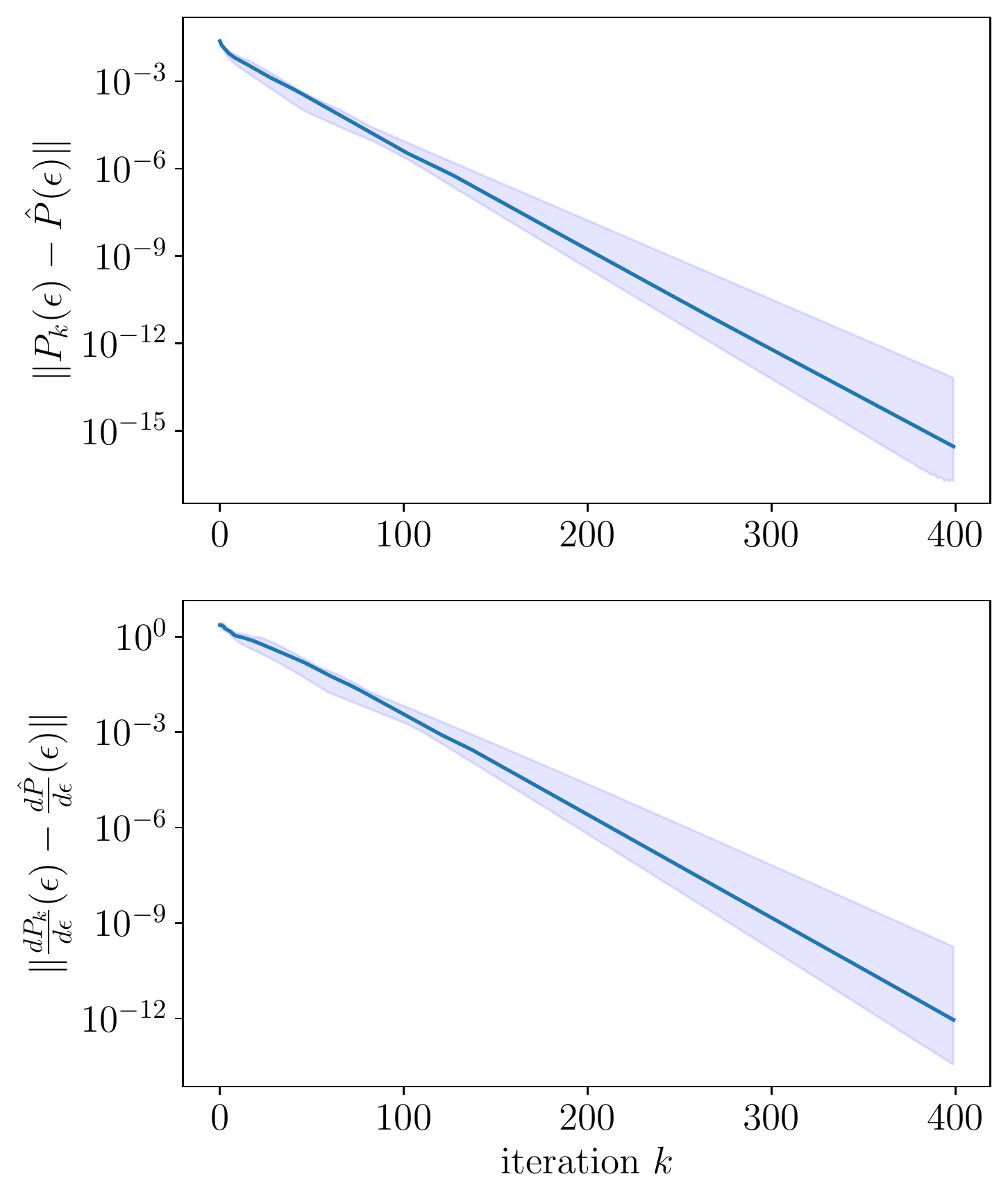}
    \captionsetup{
    format=plain,
    size=footnotesize,
    labelfont=sc,
    name=Fig.}
    \caption{Illustration of the linear convergence of the regularized transport plan $P_k(\theta)$~\eqref{eq:pk} of Sinkhorn-Knopp~\eqref{eq:defRecForward} and its derivatives $\frac{dP_k}{d\theta}(\theta)$ towards the derivative of the entropic optimal transport problem~\eqref{eq:reg-ot}.}
    \label{fig:linconv}
    \vspace{-15pt}
\end{wrapfigure}
\section{Introduction}
The optimal transport (OT) problem plays an increasing role in optimization and machine learning~\cite{peyre2019book}.
In particular, entropic regularization of OT has gained a lot of attraction by the existence of a simple and efficient algorithm introduced in~\cite{sinkhorn1967diagonal}, also known as matrix scaling or iterative proportional fitting procedure in the stochastic literature, see~\cite{ruschendorf1995convergence}.
It is known that Sinkhorn--Knopp iterates converge linearly, with an explicit rate computable from the cost matrix, to the solution of entropic OT since the work of~\cite{franklin1989scaling} introducing the use of the Hilbert metric.

\subsection{Differentiation of the Sinkhorn--Knopp algorithm}
Among the different properties of Sinkhorn--Knopp, a striking one is its differentiability with respect to the inputs.
Differentiating the iterates of the Sinkhorn--Knopp algorithm is a common routine in machine learning.
It was first used by \cite{adams2011ranking} for ranking with linear objective function.
They proposed to use backpropagation through Sinkhorn--Knopp iterates with respect to the cost matrix, without discussion of the convergence of the Jacobian.
It was later used for different applications, such as computation of Wasserstein barycenters casted as an optimization problem~\cite{bonneel2016wasserstein}, where backpropagation is performed with respect to the weight vector, for training generative models involving an OT loss as in~\cite{hashimoto2016learning,genevay2018learning}, definition of differentiable sorting procedures~\cite{cuturi2019differentiable} or solving cluster assignments problems~\cite{caron2020unsupervised}.
Popular libraries such as \texttt{POT}~\cite{flamary2021pot} or \texttt{OTT}~\cite{cuturi2022optimal} for computational optimal transport implement the backpropagation of Sinkhorn--Knopp.
To mitigate the memory footprint required by backprogation, an alternative is to use implicit differentiation as discussed first by~\cite{luise2019differential} for computing the derivatives of Sinkhorn divergences.
This approach was later used in~\cite{cuturi2020supervised,eisenberger2022unified}.
To the best of our knowledge, even though some these works justify the correctness of using automatic differentiation for a given iterate, \emph{they do not consider the issue of the convergence of the derivatives} computed by automatic differentiation. 

\subsection{Convergence of algorithmic differentiation}
The issue of the convergence of the derivatives of an algorithm was considered in the automatic differentiation community.
The linear convergence of derivatives was studied in \cite{gilbert1992automatic,griewank1993derivative} for piggyback recursion and in~\cite[Theorem 2.3]{christianson1994reverse} for backpropagation.
More recently, convergence of the derivatives of gradient descent~\cite{mehmood2020automatic,lorraine2020optimizing}, the Heavy-ball~\cite{mehmood2020automatic} method or nonsmooth fixed point methods~\cite{bolte2022automatic} were analyzed.
All these analysis \emph{require explicitly, or implicitly, that the (generalized) Jacobians are strict contractions}, \ie Lipschitz continuous with a constant strictly lesser than 1.
Unfortunately, the derivatives of Sinkhorn--Knopp do not enjoy this property.

\subsection{Contribution}
We prove (Theorem~\ref{thm:cvgt}) that the derivatives of the iterates of Sinkhorn--Knopp algorithm converge towards the derivative of the entropic regularization of optimal transport, with an explicit expression of the derivative and with a locally uniform linear convergence rate, provided that all functions entering problem definition are twice continuously differentiable.

\subsection{Organization}
Our paper is organized as follows.
Section~\ref{sec:sk} introduces the parameterized entropic regularized optimal transport problem with the Sinkhorn--Knopp algorithm and recalls linear convergence properties.
In Section~\ref{sec:dsk}, we state our main result stating the convergence of the derivatives of Sinkhorn-Knopp towards the derivatives of the regularized optimal transport with a locally uniform linear convergence rate.
Section~\ref{sec:proof} provides the proof of our result.
Section~\ref{sec:proofPerturb} contains important intermediate results toward a linear rate for the convergence.
Section~\ref{sec:lemmas} establishes miscellaneous lemmas that are used in the main proof.

\subsection{Notations}
The set of positive reals is denoted $\Rpos$, of nonnegative reals $\Rzpos$ and of nonzero reals $\RR_{\neq 0}$.
The simplex $\Delta^{n-1}$ is the set of vectors of $\Rzpos^n$ summing to 1
\[
    \Delta^{n-1} = \left\{ x \in \RR^n \, : \, \sum_{i=1}^n x_i = 1 \text{ and } x_i \geq 0, \forall i \in \{1,\dots,n\} \right\} .
\]
The identity matrix (of arbitrary size) is denoted by $I$.
For two vectors $x \in \RR^n, y \in \RR_{\neq 0}^n$, the \emph{entry-wise} (Hadamard) division $\ofrac{x}{y}$ is defined as $\left(\ofrac{x}{y}\right)_i = x_i / y_i$, and the product $x \odot y$ is defined as $(x \odot y)_i = x_i  y_i$, for all $i \in \{1, \dots, n\}$.
The 1-vector $1_n \in \RR^n$ is the vector only composed of 1's.
When the context is clear, and to lighten the notations, $\ofrac{1}{x}$ for $x \in \RR_{\neq 0}$ should be understood as $\ofrac{1_n}{x}$.
Given a function $f: \RR \to \RR$, we extend its domain as $f: \RR^p \to \RR^p$ by applying it entrywise, that is for $x \in \RR^n$, $f(x)_i = f(x_i)$, for all $i \in \{1, \dots, n\}$.
Given $l \in \NN_{>0}$ and a continuously differentiable function $F: \RR^p \to \RR^{n_1 \times \cdots \times n_l}$, we denote by $\frac{dF}{d\theta}(\theta) \in \RR^{n_1 \times \cdots \times n_l \times p}$ its Jacobian matrix (or tensor) at $\theta \in \RR^p$, \ie
\[
    \left( \frac{dF}{d\theta}(\theta) \right)_{i_1,\cdots,i_l,j}
    =
    \lim_{h \to 0} \frac{F_{i_1,\cdots,i_l}(\theta + h e_j) - F_{i_1,\cdots,i_l}(\theta)}{h} ,
\]
where $(e_j)_{j=1,\dots,n}$ is the canonical basis of $\RR^n$.
Given a differentiable function $F: \RR^n \times \RR^p \to \RR^m$, we denote by $J_F(x, \theta)$ the total derivative at $(x, \theta) \in \RR^n \times \RR^p$, that is
\[ J_F(x, \theta) =
\begin{pmatrix}
    \frac{\partial F(\cdot, \theta)}{\partial x}(x) & \frac{\partial F(x, \cdot)}{\partial \theta}(\theta)
\end{pmatrix} ,
\]
where $\frac{\partial F(\cdot, \theta)}{\partial x}(x)$ and $\frac{\partial F(x, \cdot)}{\partial \theta}(\theta)$ are the partial derivatives of $F$.

\section{Entropic optimal transport and Sinkhorn--Knopp algorithm}\label{sec:sk}

\subsection{Entropic regularization}
We consider a parametric formulation of the entropic OT\footnote{We recover the standard formulation letting $a,b,C,\epsilon$ be constant functions.}.
The entropic regularization of optimal transport associated to the parameterized marginals $a: \RR^p \to \Delta^{n-1} \cap \Rpos^n$ and $b: \RR^p \to \Delta^{n-1} \cap \Rpos^m$ of level $\epsilon: \RR^p \to \Rpos$ for the parameterized cost matrix $C: \RR^p \to \RR^{n \times m}$ reads for $\theta \in \RR^p$
\begin{equation}\label{eq:reg-ot}\tag{\ensuremath{\text{OT}_\theta}}
    \soll = \arg\min_{P \in U(\theta)} \cost(P, \theta) \eqdef \langle P, C(\theta) \rangle - \epsilon(\theta) \Ent(P) ,
\end{equation}
where $\langle P, P' \rangle = \sum_{i,j} P^{}_{i,j} P'_{i,j}$, $U(\theta)$ is the set of admissible couplings (also called transportation polytope)
\begin{equation*}
    U(\theta) \eqdef \{ P \in \Rzpos^{n \times m} \, : \, P 1_m = a(\theta) \quad \text{and} \quad P^\top 1_n = b(\theta) \} ,
\end{equation*}
and $\Ent$ is the entropic regularization of the coupling matrix $P$ defined as
\begin{equation*}
    \Ent(P) \eqdef -\sum_{i=1}^{n} \sum_{j=1}^{m} P_{i,j} \left( \log (P_{i,j}) - 1 \right) ,
\end{equation*}
where $P_{i,j} \log (P_{i,j}) = 0$ if $P_{i,j} = 0$, by continuous extension.
Note that $\cost_\theta = \cost(\cdot, \theta)$ defined in \eqref{eq:reg-ot} is $\epsilon(\theta)$-strongly convex, hence \eqref{eq:reg-ot} has a unique minimizer\footnote{The (strict) positivity follows from assumptions $a(\theta) > 0$ and $b(\theta) > 0$. Indeed, $P = a(\theta) b(\theta)^T$ is feasible for \eqref{eq:reg-ot}, with strictly positive entries, therefore Slater's qualification condition holds for \eqref{eq:reg-ot} and the required form follows from necessary and sufficient KKT conditions for the (attained) optimum, see for example \cite[Lemma 2]{cuturi2013sinkhorn}} $\soll \in \Rpos^{n\times m}$.
We will assume that all functions entering problem definition are twice continuously differentiable.

\subsection{Sinkhorn--Knopp algorithm}
The Sinkhorn--Knopp algorithm is built upon the fact~\cite[Theorem 1]{sinkhorn1964relationship} that the unique solution $\soll$ of~\eqref{eq:reg-ot} has the form for all $i \in \{1, \dots, n\}$ and $j \in  \{1, \dots, m \}$
\begin{equation}\label{eq:sol-form}
    \soll_{i,j} = u_i(\theta) K_{i,j}(\theta) v_j(\theta)
    \quad \text{where} \quad
    K_{i,j}(\theta) = \exp\left({-\frac{C_{i,j}(\theta)}{\epsilon(\theta)}}\right) > 0,
\end{equation}
for positive numbers $u_i(\theta),v_j(\theta)$, $i = 1,\ldots, n$, and $j = 1,\ldots,m$.
The goal is thus to find positive vectors $u(\theta) \in \Rpos^n$ and $v(\theta) \in \Rpos^m$, such that 
\begin{align}\label{eq:optiCondition}
    \diag(u(\theta)) K(\theta) \diag(v(\theta)) 1_m = a(\theta)
    \quad \text{and} \quad 
    \diag(v(\theta)) K(\theta)^T \diag(u(\theta)) 1_n = b(\theta).
\end{align}
In its most elementary formulation, the Sinkhorn--Knopp algorithm, also called matrix scaling problem algorithm, has the following alternating updates,
\begin{equation}\label{eq:sinkhorn}
    u_{k+1}(\theta) =  \frac{a(\theta)}{K(\theta) v_k(\theta)}
    \qquad \text{and} \qquad
    v_{k+1}(\theta) = \frac{b(\theta)}{K(\theta)^T u_{k+1}(\theta)} ,
\end{equation}
starting from a couple $(u_0(\theta), v_0(\theta)) \in \Rpos^n \times \Rpos^m$, see~\cite{thornton2022rethinking} for a discussion on initializations strategies.
Even though in practice it is not necessary to evaluate it at each iteration, one can use~\eqref{eq:sol-form} to form a current guess at iteration $k$ as $\diag (u_{k}(\theta))  K(\theta) \diag(v_{k}(\theta))$.

\subsection{Reduced formulation of Sinkhorn--Knopp}\label{sec:optimalPlan}
We will analyse an equivalent version of \eqref{eq:sinkhorn} by considering a single iterate $u$ and performing the change of variable $x = \log(u)$. Given an initilization $x_0(\theta) \in \RR^n$, this results in rewriting~\eqref{eq:sinkhorn} as the recursion in the ``log-domain''
\begin{align}\tag{\ensuremath{\text{SK}_\theta}}
    x_{k+1}(\theta) = F(x_k(\theta), \theta)
    \label{eq:defRecForward}
\end{align}
where
\begin{equation*}
     F(x, \theta) \eqdef \log(a(\theta)) - \log\left( K(\theta)\left( \frac{b(\theta)}{K(\theta)^Te^{x}} \right) \right).
\end{equation*}
Note that this formulation is close to the dual formulation of~\eqref{eq:reg-ot} as explained in~\cite[Remark 4.22]{peyre2019book}, but we will not need duality results along this paper.

We will work under the following standing assumption
\begin{assumption}[Data are continuously differentiable]
    Let $\Omega \subseteq \RR^p$ be a connected open set. The data in problem \eqref{eq:reg-ot}, \ie $C \colon \Omega \to \RR^{n \times m}$, $a \colon \Omega \to \Delta^{n-1} \cap \Rpos^n$, $b \colon \Omega \to \Delta^{m-1} \cap \Rpos^n$, $\epsilon \colon \Omega \to \Rpos$, and initialization $x_0 \colon \Omega \to \RR^n$, are all twice continuously differentiable functions on $\Omega$.
    \label{ass:mainAssumption}
\end{assumption}

It is possible to get back to the scaling factors $u_k(\theta)$ and $v_k(\theta)$ from the reduced variable $x_{k}(\theta)$ as
\[
    u_k(\theta) = e^{x_{k}(\theta)}
    \quad \text{and} \quad
    v_k(\theta) = \ofrac{b(\theta)}{K(\theta)^T e^{x_{k}(\theta)}} .
\]

Using the relationship~\eqref{eq:sol-form}, the optimal coupling matrix can be approximated as
\begin{align}
				P(x,\theta) = \diag(e^x) K(\theta) \diag\left( \ofrac{b(\theta)}{K(\theta)^Te^x} \right),
				\label{eq:optimalSolution}
\end{align}
and we construct transport plan estimates associated to each iterate, for all $k \in \NN$,
\begin{equation}\label{eq:pk}
     P_k(\theta) = P(x_k(\theta),\theta) .
\end{equation}
It is known that $P_k(\theta)$ converges linearly~\cite{franklin1989scaling} to the optimal transport plan $\hat P(\theta)$ for \eqref{eq:reg-ot}.
The next paragraph is dedicated to study the linear convergence of the reduced variable $x_k(\theta)$.

\subsection{Linear convergence of the centered reduced iterates}
It is known that $u_k(\theta)$ converges to a limit $\bar{u}(\theta)$, with a linear rate in the Hilbert metric~\cite{franklin1989scaling}, see also \cite[Theorem 4.2]{peyre2019book}, whereas we are concerned with the convergence of the reduced iterates in the ``log-domain''.
In order to study the convergence of $(x_k)_{k \in \NN}$, let us introduce the linear map $\cent$ which associates to $x$ its centered version:
\begin{equation}\label{eq:lcenter}
    \cent \colon
    \begin{cases}
        \RR^n &\to \RR^n \\
        x &\displaystyle\mapsto x - \left(\frac{1}{n} \sum_{i=1}^n x_i \right) 1_n .
    \end{cases}
\end{equation}

To analyze the convergence rate of Sinkhorn-Knopp algorithm, it is standard to use the Hilbert projective metric~\cite{birkhoff1957extensions} defined on $\Rpos^n$ as
\[
    \dH(u,u') = \|\log(u) - \log(u')\|_{\var} ,
\]
where $\|x\|_\var$ is the variation seminorm of $x \in \RR^n$ defined as
\begin{equation}\label{eq:fvar}
    \|x\|_\var = \max_{i=1,\dots,n} x_i - \min_{i=1,\dots,n} x_i .
\end{equation}

The next lemma shows the (local) linear convergence in $\ell^2$ norm of the centered reduced variable $\cent(x_k(\theta))$.
\begin{lemma}[Local linear convergence of $\cent(x_k(\theta))$]\label{lem:linearConvergenceX}
The centered reduced variable $\cent(x_k(\theta))$ converges linearly, locally uniformly, to $\cent(\bar{x}(\theta))$, \ie there exists $c: \Omega \to \Rpos$ and $\rho: \Omega \to (0,1)$ continuous such that  all $k \in \NN$ and $\theta \in \Omega$, 
\[ 
    \| \cent(x_k(\theta)) - \cent(\bar{x}(\theta)) \|
    \leq
    c(\theta) \rho(\theta)^k.
\]
Furthermore, $\theta \to \cent(\bar{x}(\theta))$ is continuous on $\Omega$.
\end{lemma}
\begin{proof}
    We combine the linear convergence result on $u_k(\theta)$ of~\cite{franklin1989scaling} with Lemma \ref{lem:variationPseudonorm}, following the suggestion of \cite[Remark 4.12]{peyre2019book}.

    We clarify below how to combine these arguments.
    We first show that the linear convergence of $u_k(\theta)$ is such that for all $\theta \in \Omega$ there exists $c(\theta)>0$ and $\rho(\theta) \in (0,1)$ such that for all $k \in \NN$
    \begin{align*}
        \dH(u_k(\theta),\bar{u}(\theta)) \leq c(\theta) \rho(\theta)^k ,
    \end{align*}
    and the mapping $c$ and $\rho$ are continuous.
    Indeed, \cite[Theorem 4]{franklin1989scaling} ensures that for all $k \in \NN$,
    \begin{align*}
        \dH(u_k(\theta),\bar{u}(\theta)) + \dH(v_k(\theta),\bar{v}(\theta)) \!\leq\! \frac{\kappa^2(K(\theta))^k}{1- \kappa^2(K(\theta))} (\dH(u_0(\theta),\bar{u}(\theta)) + \dH(v_0(\theta),\bar{v}(\theta))) ,
    \end{align*}
    where $\kappa(K)$ is the contraction ratio defined for $K \in \Rpos^{n\times m}$ as 
    \begin{align*}
        \kappa(K) = \frac{\vartheta(K)^{1/2} - 1}{\vartheta(K)^{1/2} + 1} < 1
        \quad \text{and} \quad
        \vartheta(K) = \max_{i,j,k,l} \frac{K_{i,k}K_{j,l}}{K_{j,k}K_{i,l}} .
    \end{align*}
    Remark that $P_k$ and $\soll$ enjoy the relation
    \[
        P_k = \diag\left(\ofrac{u_k(\theta)}{\bar{u}(\theta)}\right) \soll \diag\left(\ofrac{v_k(\theta)}{\bar{v}(\theta)}\right)
    \]
    and $\dH(\ofrac{u_k(\theta)}{\bar{u}(\theta)}, 1_n) = \dH(u_k(\theta),\bar{u}(\theta))$.
    Using~\cite[Theorem 4.2]{peyre2019book}, we deduce that 
    \begin{align*}
        \dH(u_k(\theta),\bar{u}(\theta))
        &\leq \frac{\kappa^2(K(\theta))^k}{(1- \kappa^2(K(\theta)))^2} \left(\dH(P(x_0(\theta),\theta) 1_m,a) + \dH(P(x_0(\theta),\theta)^T 1_n,b)\right) , \\
        &= c(\theta) \rho(\theta)^k ,
    \end{align*}
    where
    \begin{align*}
        c(\theta) &= \kappa^2(\theta) \frac{\dH(P(x_0(\theta),\theta) 1_m,a(\theta)) + \dH(P(x_0(\theta),\theta)^T 1_n,b(\theta))}{(1- \kappa^2(K(\theta)))^2} , \\
        \rho(\theta) &= \kappa^2(\theta) .
    \end{align*}
    Since for all $\theta$, $K(\theta) > 0$ and $K$ is continuous, we have that $\theta \mapsto \kappa^2(\theta)$ is continuous, and since $\theta \mapsto x_0(\theta)$ is assumed to be continuous on $\Omega$, $\theta \mapsto \dH(P(x_0(\theta), \theta)$ is also continuous.
    Thus, $c(\theta)$ and $\rho(\theta)$ depend continuously on the initial condition $x_0$ and problem data ($a,b,K,\epsilon$) which are all continuous functions of $\theta$. Therefore the linear convergence is actually locally uniform in $\theta$. 
    
    To conclude the proof, we need to remark that the Hilbert projective metric on $u$ corresponds to the variation seminorm after the change of variable $x = \log(u)$ so that for all $k \in \NN$ and all $\theta \in \Omega$,
    \begin{align*}
        \|x_k(\theta) -\bar{x}(\theta)\|_\var = \dH(u_k(\theta),\bar{u}(\theta)),
    \end{align*}
    and Lemma~\ref{lem:variationPseudonorm} provides
    \[
        \left\|\cent(x_k(\theta)) - \cent(\bar x(\theta))\right\|_\infty \leq \|x_k(\theta) - \bar x(\theta)\|_\var,
    \]
    which is the claimed result.
    
    Regarding the continuity, let $\theta_0 \in \Omega$, for all $\theta \in \Omega$ and all $k \in \NN$, we have
    \begin{align*}
        \dH(\bar{u}(\theta),\bar{u}(\theta_0)) &\leq \dH(\bar{u}(\theta),u_k(\theta)) + \dH(u_k(\theta),u_k(\theta_0)) + \dH(u_k(\theta_0),\bar{u}(\theta_0)) \\
        &\leq c(\theta) \rho(\theta)^k + c(\theta_0) \rho(\theta_0)^k + \dH(u_k(\theta),u_k(\theta_0)) .
    \end{align*}
    We may choose $k$ such that the first two terms are as small as desired uniformly for $\theta$ in a neighborhood of $\theta_0$. The last term is continuous in $\theta$ and evaluates to $0$ for $\theta = \theta_0$ so that reducing the neighborhood if necessary allows to choose it as small as desired, which proves continuity.
\end{proof}
Note that Lemma~\ref{lem:linearConvergenceX} does not imply the linear convergence of $(x_k(\theta))_{k \in \NN}$.
As we will see later in Lemma~\ref{lem:jacobianInvariance}, this is not an issue to our objective -- proving the convergence of the derivatives of~\eqref{eq:defRecForward} -- because derivatives of the algorithm enjoy a directional invariance which makes them equal when evaluated at $x_k(\theta)$ or $\cent(x_k(\theta))$ .

\section{Derivatives of Sinkhorn--Knopp algorithm and their convergence}\label{sec:dsk}

\subsection{Derivatives of the transport plan}
\label{sec:derivativeOptimalPlan}
Remark that for all $(x, \theta) \in \RR^n \times \Omega$, $P(x,\theta)$ is an $n \times m$ matrix.
Hence, $P(x, \cdot)$ is a map from $\RR^p$ to $\RR^{n \times m}$ and $P(\cdot, \theta)$ is a map from $\RR^n$ to $\RR^{n \times m}$.
Thus, we identify its partial derivatives with third-order tensors:
\begin{align}
    \frac{\partial P(\bar{x}(\theta),\theta)}{\partial x} &\in \RR^{n \times m \times n},\nonumber\\
    \frac{\partial P(\bar{x}(\theta),\theta)}{\partial \theta} &\in \RR^{n \times m \times p}.
    \label{eq:partialDerivativeP}
\end{align}
Left multiplication by these derivatives is considered as follows, for arguments of compatible size: for any $c \in \RR^n$, $\frac{\partial P(\bar{x}(\theta),\theta)}{\partial x} c \in \RR^{n \times m}$ and for any $M \in \RR^{n \times q}$, for some $q \in \NN$, $\frac{\partial P(\bar{x}(\theta),\theta)}{\partial x} M \in \RR^{n \times m \times q}$,  both operations being compatible with the usual identification of vectors as single rows in $\RR^{n \times 1}$. This multiplication is assumed to be compatible with the rules of differential calculus, for example, if $v \colon \RR^p \to \Rpos^n$ is $C^1$, then we have the identity, for any $\theta \in \RR^p$,
\begin{align}
    \frac{\partial }{\partial \theta} P(v(\theta),\theta) = \frac{\partial P(v(\theta),\theta)}{\partial x} \frac{dv(\theta)}{d \theta} + \frac{\partial P(\bar{x}(\theta),\theta)}{\partial \theta} \in  \RR^{n \times m \times p}.
    \label{eq:totalDerivativeP}
\end{align}
The operation is also invariant with order of products, if $M = u v^T$, then 
\[
    \frac{\partial P(\bar{x}(\theta),\theta)}{\partial x} M =
    \frac{\partial P(\bar{x}(\theta),\theta)}{\partial x} \left(u v^T\right) =
    \left(\frac{\partial P(\bar{x}(\theta),\theta)}{\partial x} u\right)v^T .
\]

\subsection{Spectral pseudo-inverse}
In order to explicitly describe the derivative of $\soll$, we will use the following notion of pseudo-inverse of a diagonalizable matrix.
\begin{definition}[Spectral pseudo-inverse~\cite{scroggs1966alternate,ben2003generalized}]
    Given a diagonalizable matrix $M \in \RR^{n \times n}$, let $M = Q D Q^{-1}$ be a diagonalization, where $Q \in \RR^{n \times n}$ is invertible and $D \in \RR^{n \times n}$ is diagonal.
    The \emph{spectral pseudo-inverse} of $M$ is given by $M^\sinverse = Q D^{\dagger} Q^{-1}$ where $\dagger$ denotes Moore-Penrose pseudo-inverse.
    \label{def:spectralInverse}
\end{definition}
The Moore-Penrose $D^\dagger$ pseudo-inverse of a diagonal matrix $D \in \RR^{n \times n}$ is given by $(D^\dagger)_{ii} = (D_{ii})^{-1}$ if $(D_{ii}) \neq 0$ and 0 otherwise. 
The key property of the spectral pseudo-inverse is that it preserves the eigenspaces of $M$, contrary to the more standard Moore--Penrose pseudo-inverse which preserve eigenspaces only in special cases such as symmetric matrices.
\begin{lemma}[Eigenspaces presevation of spectral pseudo-inverse~\cite{scroggs1966alternate}]
    Let $M \in \RR^{n \times n}$ a diagonalizable matrix.
    Then, $M$ and $M^\sinverse$ have the same kernel and the remaining eigenspaces are the same with inverse eigenvalues.
\end{lemma}
Note that this definition and result are defined even for non-diagonalizable matrices in~\cite{scroggs1966alternate} using its Jordan reduced form, but for the sake of our results, we only need this property for diagonalizable matrices.

\subsection{Main result}
Our contribution is the following.
\begin{theorem}[The derivatives of Sinkhorn--Knopp converge]\label{thm:cvgt}
    Under Assumption \ref{ass:mainAssumption}, let $\bar x(\theta)$ the limit of Sinkhorn--Knopp iterations~\eqref{eq:defRecForward} initialized by $x_0(\theta)$ for all $\theta \in \Omega$.
    
    Then, the optimal coupling matrix $\sol$ is continuously differentiable and its derivative $\frac{d \soll}{d\theta} \in \RR^{n \times m \times p}$ is given by 
    \begin{align*}
        \frac{d \soll}{d\theta} &=   \frac{\partial P(\bar{x}(\theta),\theta)}{\partial x} (I - A(\theta))^{\sinverse} B(\theta)  + \frac{\partial P(\bar{x}(\theta),\theta)}{\partial \theta}
    \end{align*}
    where $A(\theta)$, $B(\theta)$ are the components of the total derivative of $F$ at $(\bar{x}(\theta), \theta)$, \ie
    \begin{align*}
        [A(\theta)\; B(\theta)] &= J_F(\bar{x}(\theta), \theta),
    \end{align*}
    $F$ (resp. $P$) is defined in \eqref{eq:defRecForward} (resp. \eqref{eq:optimalSolution}), and partial derivatives of $P$ are described in Section \ref{sec:derivativeOptimalPlan}. Here $\sinverse$ denotes the spectral pseudo-inverse of a diagonalizable matrix (Definition \ref{def:spectralInverse}). 
    
    Furthermore, $P_k$ is continuously differentiable for all $k$ and the sequence of derivatives $\frac{d P_k}{d \theta}$ converges at a linear rate, locally uniformly in $\theta$. 
    In particular, for all $\theta \in \Omega$,
    \[\boxed{
    \lim_{k \to +\infty} \frac{d P_k}{d \theta}(\theta)
    =
    \frac{d \sol}{d \theta}(\theta) .}
    \]
\end{theorem}

\begin{remark}[Relation to previous works]
The differentiability of the Sinkhorn--Knopp iterations is an elementary and well-known fact, used for example in~\cite{adams2011ranking}, the new contribution here being that the derivatives converge toward the derivative of entropic regularization~\eqref{eq:reg-ot}.
Using an alternative formulation (in the context of implicit differentiation), \cite{eisenberger2022unified} proves the differentiability of the entropic regularization of OT (first part of Theorem~\ref{thm:cvgt}), and obtained an alternative expression of the derivative.
They do not however prove the convergence of the derivatives, that is the main concern of our work and the expression for the derivative in Theorem~\ref{thm:cvgt} was not mentioned in previous literature, to our knowledge.

If $F$ was a strict contraction mapping, applying~\cite[Proposition 1]{gilbert1992automatic} would be sufficient to conclude and obtain the same expression as in Theorem \ref{thm:cvgt} with an inverse instead of the spectral pseudo-inverse.
This is unfortunately not the case, and a more refined analysis is necessary to obtain the convergence. The main intuition behind this analysis is that Sinkhorn iterations are equivariant with respect to scaling of $u = \exp(x)$, and the optimal solution $P$ in \eqref{eq:optimalSolution} is invariant with respect to the same scaling. In terms of derivative, it produces a lack of invertibility of $\frac{\partial F(x,\theta)}{\partial x}$ but the corresponding direction does not depend on $(x,\theta)$, and precisely lies in the kernel of $\frac{\partial P(x,\theta)}{\partial x}$ for all $(x,\theta)$. This ``alignment'' allows to maintain an overall convergence of derivatives. Section~\ref{sec:proof} is dedicated to prove this intuition rigorously.
\end{remark}

\begin{remark}[Limitations of our result]
    Despite the generality of Theorem~\ref{thm:cvgt}, we would like to point out two limitations:
    \begin{enumerate}
        \item We do \emph{not} have any guarantees for the convergence of the derivatives of the iterates $x_k(\theta)$, $k\in \NN$.
        Said otherwise, we have guarantees for the derivatives of the optimal transport plan $P_k$, not for the derivatives of the scaling factors $u_k, v_k$, or the derivatives of the reduced variable $x_k$.
        \item Inspecting the proof of Theorem \ref{thm:cvgt}, the linear convergence factor is a $(\bar{\rho})^{\frac{1}{2}}$ where $\bar{\rho}$ is an upper bound on both the linear convergence factor of the iterates (Lemma \ref{lem:linearConvergenceX}) and the second largest eigenvalue of $\frac{\partial F}{\partial x}$ at the solution, call it $\lambda$. Classical discrete dynamical system arguments (see \cite[Remark 4.5]{peyre2019book} on local linear convergence) suggest that the linear convergence factor of the iterates is asymptotically of order $\lambda$. Taking this into consideration, our proof suggest an asymptotic linear convergence factor of the order $\sqrt{\lambda}$ for the derivatives, a factor strictly greater than that of the sequence. This discrepancy is a consequence of Lemma \ref{lem:realSequence2} which we use for simplicity of the presentation which requires a \emph{non-asymptotic} analysis to ensure uniformity in $\theta$. However, removing uniformity, this could be improved to obtain pointwise an asymptotic linear convergence factor arbitrarily close to $\lambda$ using Lemma \ref{lem:realSequence3} instead, combined with arguments outlined in \cite[Remark 4.5]{peyre2019book}. 
    \end{enumerate}
\end{remark}

\begin{remark}[Application to automatic differentiation of Sinkhorn--Knopp]
    Given $k \in \NN$ and $\dot{\theta} \in \RR^p$, forward automatic differentiation~\cite{wengert1964simple} allows to evaluate $\dot P_k = \frac{dP_k(\theta)}{d\theta} \dot{\theta} \in \RR^{n \times m}$, \eg Jacobian-Vector Products (JVP), just by implementing~\eqref{eq:defRecForward} in a dedicated framework.
    Similarly, given $\bar w_k \in \RR^{n \times m}$, the reverse mode of automatic differentiation~\cite{linnainmaa1976taylor}, also called backpropagation, computes $\bar{\theta}_k^T = \bar{w}_k^T \frac{dP_k(\theta)}{d\theta} \in \RR^p$, \eg a Vector-Jacobian Product (VJP).
    Using a similar argument as in~\cite{bolte2022automatic}, it is possible, thanks to Theorem~\ref{thm:cvgt}, to prove the convergence of these quantities.
    Note that in practice, the object of interest is not necessarly $P_k$ by itself, but its composition by another function, \eg $\langle C(\theta), P_k(\theta)\rangle$ to compute the primal Sinkhorn divergence, $\langle C(\theta), P_k(\theta)\rangle - \Ent(P_k(\theta))$ to compute the OT loss, a sum of similar terms when dealing with Wasserstein barycenters~\cite{agueh2011barycenters}, or any function $L(P_k(\theta))$ where $L: \RR^{n \times m} \to \RR^k$ is a continuously differentiable function.
    Applying our result (Theorem~\ref{thm:cvgt}) and the chain rule leads to the same convergence of automatic differentiation for such quantities.
\end{remark}

\begin{remark}[Differentiation with respect $a$, $b$, $C$ or $\epsilon$]
    Theorem~\ref{thm:cvgt} is presented with an abstract parameterization of the problem with variable $\theta \in \RR^q$. Choosing different values for $\theta$ allows to obtain derivatives of $P_k$ for $k \in \NN$ as well as $\hat{P}$ with respect to the original transport problem data: $a$, $b$, $C$ or $\epsilon$. These are typically evaluated numerically by algorithmic differentiation, but one could get closed form expressions in simple cases. For example choosing $\theta = a$, we have 
    \begin{align*}
        \frac{\partial F(x,\theta)}{\partial a} = \diag\left(\frac{1}{a}\right).
    \end{align*}
    Similarly, setting $\theta = b$, we have 
    \begin{align*}
        \frac{\partial F(x,\theta)}{\partial b} = - \diag\left(\frac{1}{K \frac{b}{K^Te^x}}\right) K \diag\left(\frac{1}{K^Te^x}\right).
    \end{align*}
    One could also compute derivatives with respect to the cost matrix $C$ or $\epsilon$, but the corresponding expressions become more complicated, and the use of automatic differentiation alleviates this difficulty in practice.
\end{remark}

\begin{remark}[Numerical illustration]
    Figure~\ref{fig:linconv} illustrates a simple example where $C$ is an Euclidean cost matrix between two point clouds $X, Y$ in $\RR^2$ of size $n_X = 100$ and $n_Y = 50$.
    The starting point cloud $X$ follows a uniform law in the square $[-1/2,1/2]$ and the target $Y$ a uniform law on a circle inscribed in the square.
    The marginals are two uniform histograms $a = 1_n/n$ and $b = 1_m/m$.
    Sinkhorn--Knopp algorithm~\eqref{eq:defRecForward} is automatically differentiated with the Python library \texttt{jax}~\cite{jax2018github} with respect to the parameter $\epsilon$, and we record the median of 10 trials for $\epsilon = 10^{-2}$.
    The blue filled area represents the first and last deciles.
    We run the algorithm for a high number of iterations $N_{\text{it}}$ and display both
    \[
        \left\| P_k(\epsilon) - \sol(\epsilon) \right\|
        \quad \text{and} \quad
        \left\| \frac{dP_k}{d \epsilon}(\epsilon) - \frac{d \sol}{d \epsilon}(\epsilon) \right\|.
    \]
    Note we assume here that $P_{N_{\text{it}}}(\epsilon)$ (resp. $\frac{dP_{N_{\text{it}}}}{d \epsilon}(\epsilon)$) is close enough the optimal solution $\sol(\epsilon)$ (resp.  $\frac{d \sol}{d \epsilon}(\epsilon)$) such that it is a good proxy. In particular, we ran~\eqref{eq:defRecForward} up to machine precision.
\end{remark}

\section{Proof of Theorem~\ref{thm:cvgt}}\label{sec:proof}
Before diving into the proof, we are going to provide important spectral properties of the Jacobians of the algorithm and transport plan (Section~\ref{sec:eigentransport}), then introduce a proxy $G$ for the Jacobian of $F$ that is a contraction mapping in contrast of $\frac{dF}{dx}$ (Section~\ref{sec:defg}) and finally rewrite~\eqref{eq:totalDerivativeP} thanks to $G$ (Section~\ref{sec:prelecomp}).

\subsection{Eigendecomposition of the transport plan and Jacobian}\label{sec:eigentransport}
The following lemma provides important properties of the Jacobians of $P$ and $F$ as a function of $x$. Here $\theta$ is fixed and we look at properties of the derivative with respect to $x$, hence the dependency in $\theta$ does not appear.
\begin{lemma}[Expression of the Jacobian of $F(x)$]
    \label{lem:computeJacobian}
    Let $x \in \RR^n$.
    \begin{enumerate}
        \item We have $\frac{dP(x)}{dx}1_n = 0_{n \times m}$, where the product is described in Section \ref{sec:derivativeOptimalPlan}. 
        \item The Jacobian $\frac{dF(x)}{dx}$ of $F$ reads
        \begin{align*}
            \frac{dF(x)}{dx} &=
            \diag\left( \frac{1}{K\left( \frac{b}{K^Te^x} \right)} \right) K \diag\left( \frac{b}{\left( K^Te^x \right)^2}  \right)K^T \diag\left( e^x \right)\\
            &=
             \diag\left(e^{F(x)}\right)\diag\left(\frac{1}{a \odot e^x}\right) P(x)\diag\left(\frac{1}{b}\right) P^T(x).
        \end{align*}
    \end{enumerate}
\end{lemma}
\begin{proof}\hfill
   \begin{enumerate}
       \item We note that $P(x + \lambda 1_n) = P(x)$ for all $\lambda \in \RR$ so that $(P(x + \lambda 1_n) - P(x))/ \lambda = \lambda \frac{dP(x)}{dx} 1_n + o(\lambda) = 0$. This implies that $\frac{dP(x)}{dx} 1_n = 0$.
       \item The first expression is a direct computation observing that if $f: \RR^n \to \RR^n$ is an entry-wise function, then $J_f(x) = \diag(f'(x))$ where $f'$ is again applied entry-wise.
       Indeed, we have for $x \in \RR^n$, $\frac{d e^x}{dx} = \diag(e^x)$, which in turns gives $\frac{d K^T e^x}{dx} = K^T \diag(e^x)$.
       Then, we obtain the derivatives of the ratio
       \[
            \frac{d \ofrac{b}{K^T e^x}}{dx}(x) = - \diag\left(\ofrac{b}{(K^T e^x)^2}\right) K^T \diag(e^x) .
       \]
       Similarly, since $K$ is a linear operator,
       \[
        \frac{d \left(K\ofrac{b}{K^T e^x}\right)}{dx}(x) = - K \diag\left(\ofrac{b}{(K^T e^x)^2}\right) K^T \diag(e^x) .
       \]
       Finally, since $\frac{d \log(g(x))}{dx} = \frac{d g(x)}{dx} \odot \frac{1}{g(x)}$, for a differentiable $g: \RR^n \to \RR^n$, we obtain that
       \[\frac{dF(x)}{dx} =
            \diag\left( \frac{1}{K\left( \ofrac{b}{K^Te^x} \right)} \right) K \diag\left( \ofrac{b}{\left( K^Te^x \right)^2}  \right)K^T \diag\left( e^x \right) .
        \]
       The second expression uses the definition of $P$ in \eqref{eq:optimalSolution}.
       Observe that
       \begin{equation}\label{eq:calc-DF-one}
        \diag\left( \ofrac{b}{\left( K^Te^x \right)^2}  \right)K^T \diag\left( e^x \right)
        =
        \diag\left(\frac{1}{K^Te^x}\right)P^T(x) ,           
       \end{equation}
       and (using the fact that diagonal matrices commute)
       \begin{equation}\label{eq:calc-DF-two}
        \diag(e^{F(x)})
        =
        \diag\left( \frac{1}{K\left( \ofrac{b}{K^Te^x} \right)} \right) \diag(a) .
       \end{equation}
       Observe now that,
       \begin{equation}\label{eq:calc-DF-three}
        K \diag\left(\frac{1}{K^Te^x}\right)
        =
        \diag\left(\frac{1}{e^x}\right) P(x) \diag\left(\frac{1}{b}\right) .
       \end{equation}
       Combining~\eqref{eq:calc-DF-one}, \eqref{eq:calc-DF-two} and \eqref{eq:calc-DF-three} gives the result.
   \end{enumerate}
\end{proof}
\begin{remark}
    \label{rem:jacobianF}
    If $x = F(x)$, at a fixed point solution, the Jacobian expression in Lemma~\ref{lem:computeJacobian} can be simplified as follows
        \begin{align*}
            \frac{dF(x)}{dx} =
            \diag\left(\frac{1}{a}\right) P(x)\diag\left(\frac{1}{b}\right) P^T(x).
        \end{align*}
\end{remark}
We have the following result on the eigenvalues and eigenvectors of $\frac{dF}{dx}$.
\begin{lemma}[Eigendecomposition of $\frac{dF}{dx}$]
    For any $x$, $\frac{dF(x)}{dx}$ is diagonalisable on $\RR$. $1$ is an eigenvalue with multiplicity $1$ and the other eigenvalue have modulus strictly smaller than $1$. Furthermore, one has the following eigenvectors:
    \begin{align*}
        \frac{dF(x)}{dx} 1_n &= 1_n \\
        \left(\frac{dF(x)}{dx}\right)^T \frac{a \odot e^x}{e^{F(x)}}&= \frac{a \odot e^x}{e^{F(x)}} 
    \end{align*}
    \label{lem:jacobianDiagonalizable}
\end{lemma}
\begin{proof}
    Fix $x \in \RR^n$ and let 
    \begin{equation*}
				S = \diag\left( \ofrac{1}{K\left( \ofrac{b}{K^Te^x} \right)} \right), \,
				M =K \diag\left( \ofrac{b}{\left( K^Te^x \right)^2}  \right)K^T, \, \text{ and } \,
				T = \diag\left( e^x \right).
    \end{equation*}
    The matrices $S$ and $T$ are diagonal with positive entries and $M$ is symmetric such that $SMT = \frac{dF(x)}{dx}$. Setting $A= (TS^{-1})^{1/2}$, we have, using the fact that diagonal matrix commute
    \begin{align*}
        A S M T A^{-1} &= T^{\frac{1}{2}}S^{-\frac{1}{2}}SMT S^{\frac{1}{2}} T^{-\frac{1}{2}} \\
        &= T^{\frac{1}{2}}S^{\frac{1}{2}}MS^{\frac{1}{2}} T^{\frac{1}{2}},
    \end{align*}
    and therefore $ A \frac{dF(x)}{dx} A^{-1}$ is real symmetric, hence diagonalisable with real eigenvalues.
    As a consequence, $\frac{dF(x)}{dx}$ being similar to $ A \frac{dF(x)}{dx} A^{-1}$ it has the same property.
    It is an easy calculation to check that $\frac{dF(x)}{dx} 1_n = 1_n$.
    Indeed, $T 1_n = e^x$, and since $\diag(y)x = y \odot x$ for $x,y \in \RR^n$, we have that $M e^x = K \ofrac{b}{K^T e^x}$ and then $S K \ofrac{b}{K^T e^x} = 1_n$.
    Multiplicity of the eigenvalue $1$ as well as properties of the remaining eigenvalue is a consequence of Perron--Frobenius theorem~\cite[Theorem 8.2.8 and Theorem 8.3.4]{horn2013matrix} applied to the stochastic matrix $\frac{dF(x)}{dx}$.
    
    Let us prove the last identity. We have
    \begin{align*}
        e^{F(x)} &= \ofrac{a}{K\left(\ofrac{b}{K^Te^x}\right)},\\
        P(x) 1_m &= \diag(e^x) K \left(\ofrac{b}{K^Te^x}\right) = \ofrac{a \odot e^x}{e^{F(x)}}, \\
        P(x)^T 1_n &=  \ofrac{b}{K^Te^x} \odot K^T e^x = b,
    \end{align*}
    from which we deduce
        \begin{align*}
        &\left(\frac{dF(x)}{dx}\right)^T \ofrac{a \odot e^x}{e^{F(x)}}\\
        &=
         P(x)\diag\left(\frac{1}{b}\right) P^T(x) \diag\left(\ofrac{e^{F(x)}}{(a \odot e^x})\right) \ofrac{a \odot e^x}{e^{F(x)}} \\
         &=  P(x)\diag\left(\frac{1}{b}\right) P^T(x) 1_n \\
         &=  P(x)1_m\\ 
         &=\ofrac{a \odot e^x}{e^{F(x)}}.
    \end{align*}
    This concludes the proof.
\end{proof}

\subsection{Reduced partial Jacobian of $F$}\label{sec:defg}
For any $(x,\theta) \in \RR^n \times \RR^p$, we set 
\begin{align}
    \alpha(x, \theta) &= 1_n^T\left(\ofrac{a(\theta) \odot e^{x}}{e^{F(x,\theta)}} \right) \nonumber \\
    v(x,\theta) &= \frac{1}{\alpha(x,\theta)} \ofrac{a(\theta) \odot e^{x}}{e^{F(x,\theta)}}.
    \label{eq:eigenvectorsJacobian}
\end{align}
For any $x,\theta$ consider furthermore the block decomposition of the total derivative of $F$,
$[A(x,\theta)\  B(x,\theta)] = J_F(x,\theta)$ and set
\begin{align}
    G(x,\theta) = A(x,\theta) - 1_n v(x,\theta)^T.
    \label{eq:reducedJacobian}
\end{align}
We call $G$ the reduced partial Jacobian of $F$.
From Lemma \ref{lem:jacobianDiagonalizable}, we have that $1_n$ is an eigenvector of $A(x,\theta)$ and $v(x,\theta)$ is an eigenvector of $A(x,\theta)^T$, both with eigenvalue $1$, which has multiplicity $1$, with $1_n^Tv(x,\theta) = 1$. Therefore Lemma \ref{lem:diagonalisableRankOnePerturbation} ensures that the matrix $G(x,\theta)$ is diagonalisable in the same basis as $A(x,\theta)$ with the same eigenvalues, except eigenvalue $1$ which is set to $0$, and therefore its spectral radius is strictly less than $1$.
Later in the proof, we will study a recursion involving $A$ (which is not a contraction), and we will use an equivalent recurrence involving $G$ (which is a contraction).
By Assumption \ref{ass:mainAssumption}, the functions $J_F, P, A,B,G$ are continuously differentiable on $\RR^n \times \Omega$.

The following lemma shows that $J_F$ and $G$ are invariant by the centering operation $\cent$, and more generally by translation of $\lambda 1_n$.
\begin{lemma}[Invariance by centering]
    For all $\lambda \in \RR$, $x \in \RR^n$, and $\theta \in \Omega$, we have,
    \begin{align*}
        F(x + \lambda 1_n, \theta) &= F(x, \theta) + \lambda 1_n ,\\
        J_F(x + \lambda 1_n,\theta) &= J_F(x,\theta) ,\\
        v(x + \lambda 1_n,\theta) &= v(x,\theta) ,\\
        G(x + \lambda 1_n,\theta) &= G(x,\theta).
    \end{align*}
    In particular, $J_F(\cent(x),\theta) = J_F(x,\theta)$ and $G(\cent(x),\theta) = G(x,\theta)$ where $\cent$ is the centering operator introduced in Lemma \ref{lem:linearConvergenceX}.
    \label{lem:jacobianInvariance}
\end{lemma}
\begin{proof}
    We have for $\lambda \in \RR$ and $x \in \RR^n$,
    \begin{align*}
        F(x + \lambda 1_n, \theta) &= 
        \log(a(\theta)) - \log\left( K(\theta)\left( \ofrac{b(\theta)}{K(\theta)^Te^{x + \lambda 1_n}} \right) \right) \\
        &= \log(a(\theta)) - \log\left( K(\theta)\left( \ofrac{b(\theta)}{e^\lambda K(\theta)^Te^{x}} \right) \right) \\
         &= \log(a(\theta)) - \log\left(e^{-\lambda} K(\theta)\left( \ofrac{b(\theta)}{ K(\theta)^Te^{x}} \right) \right) \\
        &=\log(a(\theta)) + \lambda 1_n - \log\left( K(\theta)\left( \ofrac{b(\theta)}{ K(\theta)^Te^{x}} \right) \right) \\
        &= F(x, \theta) + \lambda 1_n ,
    \end{align*}
    which implies for all $\lambda \in \RR$, $J_F(x + \lambda 1_n,\theta) = J_F(x,\theta)$.
    Observe now that
    \begin{align*}
        \ofrac{a(\theta) \odot e^{x + \lambda 1_n}}{e^{F(x+\lambda 1_n,\theta)}}
        &=
        \ofrac{a(\theta) \odot e^{\lambda}e^{x}}{e^{F(x,\theta) + \lambda 1_n}} \\
        &=
        \ofrac{a(\theta) \odot e^{\lambda}e^{x}}{e^\lambda e^{F(x,\theta)}} \\
        &=
        \ofrac{a(\theta) \odot e^{x}}{e^{F(x,\theta)}} .
    \end{align*}
    Thus,
    $\alpha(x + \lambda 1_n, \theta) = \alpha(x, \theta)$ and in turn, we get that $v(x + \lambda 1_n,\theta) = v(x,\theta)$.

    To conclude, we have
    \begin{align*}
        G(x + \lambda 1_n, \theta)
        &= A(x + \lambda 1_n,\theta) - 1_n v(x + \lambda 1_n,\theta)^T \\
        &= A(x, \theta) - 1_n v(x, \theta)^T = G(x, \theta),
    \end{align*}
    following the fact that $J_F(x + \lambda 1_n,\theta) = J_F(x,\theta)$, and in particular $A(x + \lambda 1_n,\theta) = A(x,\theta)$.
\end{proof}

\subsection{Preliminary computation}\label{sec:prelecomp}
We start with some computation and notations before providing the proof arguments.
Setting for all $k \in \NN$, and $\theta \in \RR^p$, $[A_k(\theta)\  B_k(\theta)] = J_F(x_k(\theta),\theta)$ we have the piggyback recursion
\begin{align}
    \label{def:eqDefAutodiffx}
    \frac{d x_{k+1}(\theta)}{d \theta} = A_k(\theta) \frac{d x_{k}(\theta)}{d \theta} + B_k(\theta),
\end{align}
We have for all $k$ and $\theta$, using \eqref{eq:totalDerivativeP} for the total derivative of $P$,
\begin{align}
    \frac{d P_{k+1}(\theta)}{d \theta} &= \frac{\partial P(x_{k+1}(\theta),\theta)}{\partial x} \frac{dx_{k+1}(\theta)}{d\theta} + \frac{\partial P(x_{k+1}(\theta),\theta)}{\partial \theta}\nonumber\\
    &=\frac{\partial P(x_{k+1}(\theta),\theta)}{\partial x} \left(A_k(\theta) \frac{dx_{k}(\theta)}{d\theta} +B_k(\theta) \right)+ \frac{\partial P(x_{k+1}(\theta),\theta)}{\partial \theta}.
    \label{eq:defAutodiffP}
\end{align}

For all $\theta$ and all $k \in \NN$, we have $A_k(\theta) = A(x_k(\theta),\theta)$, we set
\begin{align*}
    G_k(\theta)= G(x_k(\theta),\theta) = A_k(\theta) - 1_n v(x_k(\theta),\theta)^T,
\end{align*}
where $G$ is defined in \eqref{eq:reducedJacobian} and $v$ is defined in \eqref{eq:eigenvectorsJacobian}.
From Lemma \ref{lem:diagonalisableRankOnePerturbation}, the matrix $G_k(\theta)$ is diagonalisable in the same basis as $A_k(\theta)$ with the same eigenvalues except eigenvalue $1$ which is set to $0$ and therefore its spectral radius is strictly less than $1$.

From Lemma \ref{lem:computeJacobian}, we have $\frac{\partial P(x,\theta)}{\partial x} 1_n = 0_{n \times m}$ for all $(x,\theta)$ and therefore
\begin{align*}
    \frac{\partial P(x,\theta)}{\partial x}G_k(\theta) &=  \frac{\partial P(x,\theta)}{\partial x}A_k(\theta) - \frac{\partial P(x,\theta)}{\partial x}1_n v(x_k(\theta),\theta)^T= \frac{\partial P(x,\theta)}{\partial x}A_k(\theta).
\end{align*}
Plugging this in \eqref{eq:defAutodiffP}, we obtain
\begin{align*}
    \frac{d P_{k+1}(\theta)}{d \theta} &=\frac{\partial P(x_{k+1},\theta)}{\partial x} \left(A_k(\theta) \frac{dx_{k}}{d\theta} +B_k(\theta) \right)+ \frac{\partial P(x_{k+1},\theta)}{\partial \theta}\\
    &=\frac{\partial P(x_{k+1},\theta)}{\partial x} \left(G_k(\theta) \frac{dx_{k}}{d\theta} +B_k(\theta) \right)+ \frac{\partial P(x_{k+1},\theta)}{\partial \theta}.
\end{align*}
This allows to rewrite the iterations equivalently, with $D_0 = \frac{d x_0}{d\theta}$, for all $k \geq 0$ and $\theta$, using the product rule for partial derivatives of $P$ defined in Section \ref{sec:derivativeOptimalPlan},
\begin{align}
    \frac{d P_k(\theta)}{d \theta} &= \frac{\partial P(x_{k},\theta)}{\partial x}D_k(\theta) + \frac{\partial P(x_{k},\theta)}{\partial \theta},\nonumber\\
    D_{k+1}(\theta) &= G_k(\theta) D_k(\theta) + B_k(\theta).
    \label{eq:autodiffModified}
\end{align}

We are now ready to prove our main result.
\begin{proof}[Proof of Theorem~\ref{thm:cvgt}]\hfill

\textbf{Convergence of $A_k$, $G_k$ and $B_k$.}
For all $\theta \in \Omega$, from Lemma \ref{lem:linearConvergenceX}, the centered iterates $(\cent(x_k(\theta)))_{k \in \NN}$ converge with a linear rate to $\cent(\bar{x}(\theta))$ which is locally uniform in $\theta$. Furthermore, using Assumption \ref{ass:mainAssumption}, $F$ is twice continuously differentiable jointly in $x \in \RR^n$ and $\theta \in \Omega$ and therefore $J_F$ and $G$ are continuously differrentiable, and hence locally Lipschitz on $\RR^n \times \Omega$. 

We remark that for all $\theta$, using Lemma \ref{lem:jacobianInvariance}
\begin{align*}
    G_k(\theta) &= G(x_k(\theta),\theta) = G(\cent(x_k(\theta)),\theta), 
\end{align*}
so that, as $k \to \infty$, $G_k(\theta)$ converges with a locally uniform linear rate to $G(\theta) := G(\cent(\bar{x}(\theta)),\theta) = G(\bar{x}(\theta),\theta)$. Similarly 
$B_k(\theta)$ converges with a locally uniform linear rate to $B(\theta):= B(\bar{x}(\theta),\theta)$ and $A_k(\theta)$ converges with a locally uniform linear rate to $A(\theta):= A(\bar{x}(\theta),\theta)$. Note that by Lemma \ref{lem:linearConvergenceX}, the map $\theta \mapsto \cent(\bar{x}(\theta))$ is continuous, so that $A$, $G$ and $B$ are continuous functions of $\theta$.

For any $\theta$, $G(\theta)$ is digaonalizable with spectral radius strictly less than $1$, the recursion on $D_k(\theta)$ should converge with a locally uniformly linear rate in $\theta$. This assertion is a consequence of the following lemma which explicit the constants appearing in the linear rate for the matrix recursion.
\begin{lemma}[Explicit rate for linear convergence]\label{lem:perturbedConvergenceBis}
  Let $\rho < 1$ and $\bar G \in \RR^{n \times n}$ be diagonalisable on $\RR$, with spectral radius smaller than $\rho$ and and $Q$ an invertible matrix which rows are made of an eigenbasis of $\bar{G}$. Let $\bar B \in \RR^{n \times m}$.
  Let $(G_k)_{k \in \NN}$ and $(B_k)_{k \in \NN}$ be sequences of matrices such that there exists a constant $c_1>0$ such that for all $k \in \NN$,
  \begin{align}
    \| G_k - \bar G \|_{\mathrm{op}}  &\leq c_1 \rho^{k+1}, \label{eq:lemmaErrorHypG} \\
    \| B_k - \bar B \|  &\leq c_1 \rho^{k+1}. \label{eq:lemmaErrorHypB}
  \end{align}
  Then, for the recursion
  \begin{align*}
    D_{k+1} = G_k D_k + B_k ,
  \end{align*}
  setting $\bar{D} = (I - \bar{G})^{-1} \bar B$, there exists a continuous function $\const \colon \Rzpos^5 \times (0,1) \to \Rzpos$ such that for all $k \in \NN$,
  \begin{align*}
      \|D_k - \bar{D} \| \leq \rho^{\frac{k}{2}} \const(\|Q\|_\mathrm{op}, \|Q^{-1}\|_\mathrm{op}, c_1, \|D_0\|, \|\bar{B}\|, \rho).
  \end{align*}
\end{lemma}

\textbf{Convergence of $D_k$.} Let us explicit how Lemma \ref{lem:perturbedConvergenceBis} allows to prove convergence of $(D_k(\theta))_{k\in \NN}$. Start with a fixed $\theta \in \Omega$, we first drop the dependency in $\theta$ for clarity.
We have from Remark \ref{rem:jacobianF}
        \begin{align*}
            A =
            \diag\left(\frac{1}{a}\right) \hat{P}\diag\left(\frac{1}{b}\right) \hat{P}^T.
        \end{align*}
Setting $S =  \diag\left(\frac{1}{\sqrt{a}}\right)$, we have that 
\begin{align*}
    S^{-1} A S = \diag\left(\frac{1}{\sqrt{a}}\right) \hat{P}\diag\left(\frac{1}{b}\right) \hat{P}^T \diag\left(\frac{1}{\sqrt{a}}\right) ,
\end{align*}
which is symmetric. Therefore, there is an orthogonal matrix $U($ and diagonal matrix $E$ such that
\begin{align*}
    S^{-1} A S= U E U^T,
\end{align*}
and
\begin{align*}
    A = SU E U^T S^{-1} = SUE (SU)^{-1}.
\end{align*}
Set $Q = SU$, we have by submultiplicativity of $\|\cdot\|_{\mathrm{op}}$
\[
    \|Q\|_\mathrm{op} \leq \|U\|_\mathrm{op} \|S\|_\mathrm{op} = \|S\|_\mathrm{op} = \left\| \frac{1}{\sqrt{a}}\right\|_\infty .
\]
Similarly $\|Q^{-1}\|_\mathrm{op} = \left\| \sqrt{a}\right\|_\infty$. From Lemma \ref{lem:diagonalisableRankOnePerturbation}, $Q$ diagonalizes both $A$ and $G$.

Getting back the dependency in $\theta$, we fix $\theta_0 \in \Omega$, and set for all $\theta \in \Omega$
\begin{align*}
\bar{D} &\colon \theta \mapsto (I - G(\theta))^{-1} B(\theta) ,  \\ 
\bar{\rho} &\colon \theta \mapsto \max\{\rho(\theta),
\|Q(\theta)^{-1}G(\theta)Q(\theta)\|_\mathrm{op}\}<1 ,   
\end{align*}
where $\rho(\theta)< 1$ is given in Lemma \ref{lem:linearConvergenceX} and $\|Q(\theta)^{-1}G(\theta)Q(\theta)\|_\mathrm{op}$ is the largest eigenvalue, in absolute value, of $G(\theta)$, which is smaller than $1$ and continuous with respect to $\theta$. In particular, $\bar{\rho}$ is continuous. 

Fix a compact set $V\subset \Omega$ which contains $\theta_0$ in its interior and a compact set $W \subset \RR^n$ which contains $\cent(x_k(\theta))$ for all $k \in \NN$ and $\theta \in V$ (this exists thanks to Lemma \ref{lem:linearConvergenceX}).
We set $c_1 \colon \Omega \to \Rzpos$ such that $c_1 = Lc / \rho$ where $c \colon \Omega\to \Rzpos$ is the constant in Lemma \ref{lem:linearConvergenceX} and $L$ is a Lipschitz constant of $J_F$ and $G$ on $W \times V$ (recall that they are continuously differentiable).
Using Lemma~\ref{lem:jacobianInvariance}, we have for all $\theta \in V$ and $k \in \NN$, 
\begin{align*}
    \|J_F(x_k(\theta),\theta) - J_F(\bar{x}(\theta),\theta)\| 
    &= \|J_F(\cent(x_k(\theta)),\theta) - J_F(\cent(\bar{x}(\theta)),\theta)\| \\
    &\leq c_1(\theta) \bar{\rho}(\theta)^{k+1},
\end{align*}
and
\begin{align*}
    \|G(x_k(\theta),\theta) - G(\bar{x}(\theta),\theta)\| 
    &= \|G(\cent(x_k(\theta)),\theta) - G(\cent(\bar{x}(\theta)),\theta)\|\| \\
    &\leq c_1(\theta) \bar{\rho}(\theta)^{k+1} .
\end{align*}
The largest eigenvalue of $G(\theta)$ is at most $\bar{\rho}(\theta)$ so that Lemma \ref{lem:perturbedConvergenceBis} applies, and we have for all $k \in \NN$ and all $\theta \in V$,
\begin{align*}
    &\|D_k(\theta) - \bar{D}(\theta)\|   \\
    \leq\ &\bar{\rho}(\theta)^{\frac{k}{2}} \const\left( \left\| \frac{1}{\sqrt{a(\theta)}}\right\|_\infty,  \left\| \sqrt{a(\theta)}\right\|_\infty, c_1(\theta), \left\|\frac{d x_0(\theta)}{d\theta}\right\|, \|B(\theta)\|, \bar{\rho}(\theta)\right),
\end{align*}
where $\const \colon \Rzpos^5 \times (0,1)$ is continuous.
All terms in the right hand side are continuous functions of $\theta$ and can be uniformly bounded on $V$, so that $D_k(\theta) \to \bar D(\theta) = (I - G(\theta))^{-1} B(\theta)$ at a locally uniform linear convergence rate.

\textbf{Convergence of the derivatives of Sinkhorn--Knopp towards the derivatives of entropic regularization.}
From Lemma \ref{lem:perturbedConvergenceBis} the limit of $(D_k(\theta_0))_{k\in\NN}$ is of the form
\begin{align*}
    \bar{D}(\theta_0) &= (I - G(\theta_0))^{-1} B(\theta_0) \\
    [A(\theta_0)\; B(\theta_0)] &= J_F(\bar{x}(\theta_0), \theta_0).
\end{align*}
Recall that for any $\lambda \in \RR$, and any $x,\theta$, $P(x + \lambda 1_n, \theta) = P(x + \lambda 1_n, \theta)$ so that $J_P(x + \lambda 1_n, \theta) = J_P(x + \lambda 1_n, \theta)$. Therefore expression \eqref{eq:autodiffModified} is equivalently rewritten as
\begin{align}
    \frac{d P_k(\theta)}{d \theta} &= \frac{\partial P(\cent(x_{k}),\theta)}{\partial x}D_k(\theta) + \frac{\partial P(\cent(x_{k}),\theta)}{\partial \theta},\nonumber\\
    \label{eq:autodiffModifiedCentered}
\end{align}
We have shown locally uniform linear convergence of both $D_k(\theta)$ and $\cent(x_k(\theta))$, by Assumption \ref{ass:mainAssumption},  equation \eqref{eq:autodiffModifiedCentered} is continuously differentiable, hence it has a locally Lipschitz dependency in $\cent(x_k),D_k$ and $\theta$, so that as $k\to \infty$ uniformly linearly in a neighborhood of $\theta_0$ 
\begin{align}
    \label{eq:limitMkDerivative}
    \lim_{k \to \infty} \frac{d}{d\theta}P_{k}(\theta) = \frac{\partial P(\bar{x}(\theta),\theta)}{\partial x} \bar{D}(\theta) + \frac{\partial P(\bar{x}(\theta),\theta)}{\partial \theta} .
\end{align}
Note that $P_k(\theta)$ converges pointwise towards $\soll = P(\bar{x}(\theta), \theta)$ which is a solution to problem \eqref{eq:reg-ot}.
By local uniform convergence of derivatives and the fact that $P_k$ are continuously differentiable, thanks to Lemma~\ref{lem:cont-diff}, we have that $\sol$ is continuously differentiable and 
\begin{align*}
    \lim_{k \to \infty} \frac{dP_k(\theta)}{d\theta} = \frac{d \soll}{d\theta}.
\end{align*}

\textbf{Expression of the derivative.}
Finally, by construction of $G$ in  \eqref{eq:reducedJacobian} and thanks to Lemma \ref{lem:diagonalisableRankOnePerturbation}, we have for all $x,\theta$, that $I - A(x,\theta)$ and $I - G(x,\theta)$ have the same eigenspaces all eigenvalues being nonzero except the one generated by $1_n$ for which corresponds to eigenvalue $0$ for $I-A(x,\theta)$ and $1$ for $I - G(x,\theta)$. Therefore, we have $(I-G(x,\theta)))^{-1} = (I-A(x,\theta))^\sinverse + 1_n v(x,\theta)^T$, where $v(x,\theta)$ is the normalized eigenvector of $A(\theta)^T$ associated to eigenvalue $1$ (see \eqref{eq:eigenvectorsJacobian}). Recall that $\sinverse$ denotes the spectral pseudo-inverse for diagonalisable matrices (Definition \ref{def:spectralInverse}). From Lemma \ref{lem:computeJacobian}, we have $\frac{dP(x,\theta)}{dx} 1_n = 0$ for all $(x,\theta)$, therefore for all $\theta \in \Omega$,
\begin{align*}
     \frac{\partial P(\bar{x}(\theta),\theta)}{\partial x} \bar{D}(\theta) &=\frac{\partial P(\bar{x}(\theta),\theta)}{\partial x} (I - G(\bar{x}(\theta),\theta))^{-1} B(\theta) , \\
     &= \frac{\partial P(\bar{x}(\theta),\theta)}{\partial x} ( (I-A(\theta))^\sinverse + 1_n v(x,\theta)^T) B(\theta) \\
     & = \frac{\partial P(\bar{x}(\theta),\theta)}{\partial x} (I - A(\theta))^{\sinverse} B(\theta).
\end{align*}
We have therefore that
\begin{align*}
    \frac{d \soll}{d\theta} &=   \frac{\partial P(\bar{x}(\theta),\theta)}{\partial x} (I - A(\theta))^{\sinverse} B(\theta)  + \frac{\partial P(\bar{x}(\theta),\theta)}{\partial \theta} , \\
    [A(\theta)\; B(\theta)] &= J_F(\bar{x}(\theta), \theta) ,
\end{align*}
which concludes the proof.
\end{proof}

\section{Proof of Lemma \ref{lem:perturbedConvergenceBis}}\label{sec:proofPerturb}

We start with two lemmas on real sequences.
The first one is a quantitative version of  \cite[Lemma 9, Chapter 2]{polyak1987introduction}.

\begin{lemma}[Quantitative Gladyshev convergence]
    Let $(\alpha_k)_{k\in \NN}$ and $(\beta_k)_{k\in \NN}$ be positive summable sequences and $(z_k)_{k \in \NN}$ be a positive squence such that for all $k \in \NN$
    \begin{align*}
        z_{k+1} \leq (1 + \alpha_k)z_k + \beta_k.
    \end{align*}
    Then for all $k \in \NN$, 
    \begin{align*}
        z_k \leq \exp\left( \sum_{i=0}^{+\infty} \alpha_i\right)\left(z_0 +  \sum_{j=0}^{+\infty} \beta_j\right)
    \end{align*}
    \label{lem:realSequence1}
\end{lemma}
\begin{proof}
    For all $k \in \NN$, set
    \begin{align*}
        w_k = z_k \prod_{i=k}^{+ \infty} (1 + \alpha_i) + \sum_{i=k}^{+\infty} \beta_i \prod_{j=i+1}^{+ \infty} (1 + \alpha_j).
    \end{align*}
    Remark that using concavity of logarithm $\prod_{i=0}^{+ \infty} (1 + \alpha_i) \leq \exp\left( \sum_{i=0}^{+\infty} \alpha_i\right)$, so that $w_k$ is well defined. Remark also that $w_k \geq z_k$ for all $k$.
    
    The sequence $(w_k)_{k \in \NN}$ is decreasing, indeed, we have for all $k \in \NN$,
    \begin{align*}
        w_{k+1} &= z_{k+1} \prod_{i=k+1}^{+ \infty} (1 + \alpha_i) + \sum_{i=k+1}^{+\infty} \beta_i\prod_{j=i+1}^{+ \infty} (1 + \alpha_j)\\
        &\leq \left( (1 + \alpha_k)z_k + \beta_k\right) \prod_{i=k+1}^{+ \infty} (1 + \alpha_i) + \sum_{i=k+1}^{+\infty} \beta_i\prod_{j=i+1}^{+ \infty} (1 + \alpha_j)\\
        &= z_k \prod_{i=k}^{+ \infty} (1 + \alpha_i) + \sum_{i=k}^{+\infty} \beta_i \prod_{j=i+1}^{+ \infty} (1 + \alpha_j) \\
        &= w_k.
    \end{align*}
    Therefore, for all $k \in \NN$,
    \begin{align*}
        z_k &\leq w_k \\
        &\leq w_0  \\
        & =v_0\prod_{i=0}^{+ \infty} (1 + \alpha_i) + \sum_{i=0}^{+\infty} \beta_i \prod_{j=i+1}^{+ \infty} (1 + \alpha_j) \\
        & \leq \prod_{i=0}^{+ \infty} (1 + \alpha_i) \left(v_0 + \sum_{i=0}^{+\infty} \beta_i\right) \leq  \exp\left( \sum_{i=0}^{+\infty} \alpha_i\right) \left(v_0 + \sum_{i=0}^{+\infty} \beta_i\right), \\
    \end{align*}
    and the result follows.
\end{proof}
The following lemma specify Lemma~\ref{lem:realSequence1} when $\alpha_k$ and $\beta_k$ are geometric sequences.
\begin{lemma}[Application of Gladyshev's convergence to geometric sequences]
    Let $\rho \in (0,1)$, $c > 0$, and $(\delta_k)_{k \in \NN}$ be a positive sequence such that for all $k \in \NN$,
    \begin{align}\label{eq:halfRateAssum}
        \delta_{k+1} \leq (\rho + c \rho^{k+1}) \delta_k + c \rho^{k+1}.
    \end{align}
    Then, $(\delta_k)_{k \in \NN}$ is a geometric sequence: for all $k \in \NN$,
    \begin{align*}
        \delta_k  &\leq   \rho^{\frac{k}{2}} \exp\left( \frac{c\sqrt{\rho}}{1-\rho} \right)\left(\delta_0 +  \frac{c\sqrt{\rho}}{1-\sqrt{\rho}}\right),
    \end{align*}
    \label{lem:realSequence2}
\end{lemma}
\begin{proof}
    Dividing~\eqref{eq:halfRateAssum} on both sides by $c\rho^{(k+1)/2}$, we have for all $k \in \NN$,
    \begin{align*}
        \frac{\delta_{k+1}}{c\rho^{\frac{k+1}{2}}} &\leq \frac{\delta_k}{c\rho^{\frac{k-1}{2} }} + \frac{\delta_k}{c\rho^{\frac{k}{2} }}  \frac{c \rho^{k+1} c\rho^{\frac{k}{2}} }{c \rho^\frac{k+1}{2}} + \rho^{\frac{k+1}{2}}  \\
        &=\frac{\sqrt{\rho}\delta_k}{c\rho^{\frac{k}{2} }} + \frac{\delta_k}{c\rho^{\frac{k}{2} }}  c \rho^{k + \frac{1}{2}} + \rho^{\frac{k+1}{2}} \\
        &\leq  \frac{\delta_k}{c\rho^{\frac{k}{2}}} (1 + c\rho^{k+ \frac{1}{2}})+ \rho^{\frac{k+1}{2}} .
    \end{align*}
    Setting for all $k \in \NN$,
    \begin{align*}
        z_k = \frac{\delta_k}{c\rho^{\frac{k}{2}}} , \quad
        \alpha_k = c \rho^{k+ \frac{1}{2}} , \quad \text{and} \quad
        \beta_k = \rho^{\frac{k+1}{2}},
    \end{align*}
    we may apply Lemma \ref{lem:realSequence1} to obtain the result. Note that $\sum_{i=0}^{+\infty} \alpha_i = \frac{c\sqrt{\rho}}{1-\rho}$ and $\sum_{i=0}^{+\infty} \beta_i = \frac{\sqrt{\rho}}{1-\sqrt{\rho}}$ , so that for all $k \in \NN$
    \begin{align*}
        \frac{\delta_k}{c\rho^{\frac{k}{2}}}  = z_k
        & \leq \exp\left( \sum_{i=0}^{+\infty} \alpha_i\right)\left(z_0 +  \sum_{j=0}^{+\infty} \beta_j\right) \\
        &= \exp\left( \frac{c\sqrt{\rho}}{1-\rho} \right)\left(\frac{\delta_0}{c}  +  \frac{\sqrt{\rho}}{1-\sqrt{\rho}}\right),
    \end{align*}
    which is the desired result.
\end{proof}

\begin{lemma}[Reduced perturbated convergence]\label{lem:perturbedConvergenceBisZero}
  Let $\rho < 1$ and $\bar G \in \RR^{n \times n}$ have operator norm smaller than $\rho$ and $\bar B \in \RR^{n \times m}$.
  Let $(G_k)_{k \in \NN}$ and $(B_k)_{k \in \NN}$ be a sequence of matrices such that there exists a constant $c_0>0$ such that for all $k \in \NN$,
  \begin{align*}
    \| G_k - \bar G \|_{\mathrm{op}}  &\leq c_0 \rho^{k+1}, \\
    \| B_k - \bar B \|  &\leq c_0 \rho^{k+1}.
  \end{align*}
  Then for the recursion
  \begin{align*}
    D_{k+1} = G_k D_k + B_k ,
  \end{align*}
  setting $\bar{D} = (I - \bar{G})^{-1} \bar B$, we have
  \begin{align*}
    &\| D_k - \bar{D} \| \\
    \leq \quad & \rho^{\frac{k}{2}} \exp\left( c_0 \sqrt{\rho}\frac{1+ \|\bar B\|}{(1 - \rho)^2} \right)\left(\|D_0\| +\frac{\|\bar B\|}{1 - \rho}  +  \frac{c_0 \sqrt{\rho} (1 + \|B\|)}{(1-\sqrt{\rho})^2}\right).
  \end{align*}
\end{lemma}
\begin{proof}
    Note that $\bar{G}$ is invertible and it follows that the potential limit is $\bar{D} = (I - \bar{G})^{-1} \bar B $, as it is a fixed point of the limiting recursion, $\bar D = \bar G \bar D + \bar B$.
    We rewrite the recursion as follows
    \begin{align*}
         D_{k+1} - \bar{D} &= G_k D_k+ B_k - \bar G \bar D - \bar{B} \\
         &= G_k (D_k - \bar D) +  (G_k - \bar G)\bar D +B_k  - \bar{B}.
    \end{align*}
    Setting for all $k \in \NN$, $\delta_k = \|D_k - \bar{D}\|$, using the fact that $\|\cdot\|_\mathrm{op}$ is subordinate to $\|\cdot\|$, we have the recursion,
    \begin{align*}
        \delta_{k+1} &\leq \|G_k (D_k - \bar D)\| +  \|(G_k - \bar G)\bar D\| + \|B_k  - \bar{B}\|\\
        &\leq \|G_k\|_\mathrm{op}\| (D_k - \bar D)\| +  \|(G_k - \bar G)\|_\mathrm{op}\|\bar D\| + \|B_k  - \bar{B}\|\\
        &\leq (\rho + c_0 \rho^{k+1}) \delta_k + c_0 \rho^{k+1} (\|\bar{D}\| + 1). 
    \end{align*}
    Note that $ \|\bar{D}\| = \|(I - \bar{G})^{-1} \bar B \| \leq \|(I - \bar{G})^{-1}\|_{\mathrm{op}} \| \bar B \| \leq \frac{\|\bar B\|}{1 - \rho}$.
    Since
    \begin{align*}
        c_0 \leq c_0  \frac{1 + \|\bar B\|}{1 - \rho}
        \quad \text{and} \quad
        c_0 \left( 1+  \frac{\|\bar B\|}{1-\rho}\right) \leq c_0 \frac{1 + \|\bar B\|}{1 - \rho} ,
    \end{align*}
    we apply Lemma \ref{lem:realSequence2} with $c =  c_0  \frac{1 + \|\bar B\|}{1 - \rho}$ and use the fact that $\frac{1}{1-\rho} \leq \frac{1}{1-\sqrt{\rho}}$ and $\|D_0 - \bar D\| \leq \|D_0\| +\frac{\|\bar B\|}{1 - \rho}$.
\end{proof}

\begin{proof}[Proof of Lemma \ref{lem:perturbedConvergenceBis}]
    Note that $\bar{G}$ is invertible and it follows that the potential limit is $\bar{D} = (I - \bar{G})^{-1} \bar B $, which satisfy $\bar D = \bar A \bar D + \bar B$.
    Since $\bar G$ is diagonalisable in the basis given by $Q$, there is a diagonal matrix $E$ such that $\bar G = Q E Q^{-1}$. We rewrite equivalently the recursion as follows
    \begin{align*}
        Q^{-1} D_{k+1} = Q^{-1} G_k Q Q^{-1}D_k + Q^{-1} B_k ,
    \end{align*}
    and setting $\tilde{D}_k = Q^{-1} D_k$, $\tilde{G}_k = Q^{-1} G_k Q$ and $\tilde{B}_k = Q^{-1} B_k $ for all $k \in \NN$, this reduces to
    \begin{align*}
        \tilde{D}_{k+1} = \tilde{G}_k  \tilde{D}_k + \tilde{B}_k.
    \end{align*}
    When $k \to \infty$, we have $\tilde{G}_k \to E$, which has operator norm at most $\rho$ and $\tilde{B}_k \to Q^{-1} \bar B$. Set $\tilde{D}$ the fixed point of the limiting recursion for $\tilde{D}_k$,
    \begin{align*}
        \tilde{D} = (I- E)^{-1} Q^{-1} \bar B = Q^{-1} Q (I- E)^{-1} Q^{-1} \bar B =Q^{-1} (I- QEQ^{-1})^{-1} \bar B    = Q^{-1} \bar D.
    \end{align*}
    Furthermore for all $k \in \NN$, we have the following bounds
    \begin{align*}
        \| \tilde{G}_k - E \|_{\mathrm{op}} 
        &= \| Q^{-1} (G_k - \bar{G}) Q \|_{\mathrm{op}} & \\
        & \leq \| Q^{-1}\|_\mathrm{op}\| (G_k - \bar{G}) \|_\mathrm{op}\|Q \|_{\mathrm{op}} & {\footnotesize\text{($\|\cdot\|_{\mathrm{op}}$ is submultiplicative)}} \\
        & \leq  \left(c_1 \| Q^{-1}\|_\mathrm{op} \| Q\|_\mathrm{op}\right) \rho^{k+1} & 
        {\footnotesize\text{(by hypothesis \eqref{eq:lemmaErrorHypG})}} \\
        \| \tilde{B}_k -  Q^{-1} \bar B\| &= \| Q^{-1} (B_k - \bar{B})  \|\\
        &\leq \| Q^{-1}\|_\mathrm{op} \|B_k - \bar{B}  \| & {\footnotesize\text{($\|\cdot\|_{\mathrm{op}}$ is subordinate to $\|\cdot\|$)}}  & 
        \\
        &\leq \left( c_1 \| Q^{-1}\|_\mathrm{op}\right) \rho ^{k+1}, & 
        {\footnotesize\text{(by hypothesis~\eqref{eq:lemmaErrorHypB})}} \\
        \| Q^{-1} \bar B\| &= \| Q^{-1}\|_\mathrm{op} \|\bar{B}\|\\
        \| \tilde{D}_0 \| &= \| Q^{-1}\|_\mathrm{op} \|D_0\|.
    \end{align*}
    We apply Lemma \ref{lem:perturbedConvergenceBisZero} with 
    \begin{align*}
        c_0 = c_1 \| Q^{-1}\|_\mathrm{op}  (1 + \|Q\|_\mathrm{op}) ,
    \end{align*}
    which gives for all $k \in \NN$,
    \begin{align*}
        &\|D_k - \bar D\| \\
        =\ & \|Q (\tilde{D}_k - \tilde{D})\| \\
        \leq\ &\|Q \|_\mathrm{op} \|\tilde{D}_k - \tilde{D}\| \\
        \leq\ & \rho^{\frac{k}{2}} \|Q \|_\mathrm{op} \exp\left(  c_1 \| Q^{-1}\|_\mathrm{op}  (1 + \|Q\|_\mathrm{op}) \sqrt{\rho}\frac{1+ \| Q^{-1}\|_\mathrm{op}\|\bar B\|}{(1 - \rho)^2} \right) \\
        &\quad\times \| Q^{-1}\|_\mathrm{op}\left(\|D_0\| +\frac{\|\bar B\|}{1 - \rho}  +  \frac{c_1  (1 + \|Q\|_\mathrm{op}) \sqrt{\rho} (1 + \|\bar B\|)}{(1-\sqrt{\rho})^2}\right) ,
    \end{align*}
    which is the desired result
\end{proof}

\section{Additional lemmas}\label{sec:lemmas}
In the following, we prove some technical, but important, lemmas used in the main proof.
\begin{lemma}[Reduced eigenspace]
    Let $A \in \RR^{n \times n}$ be diagonalisable. Let $u$ be such that $Au = u$ and $v$ such that $A^Tv = v$, and assume that eigenvalue $1$ is simple, and that $u v^T = 1$. Then $\tilde{A} := A - uv^T$ and $A$ have the same eigenspaces with the same eigenvalues, except eigenvalue $1$ for $A$ which is set to $0$ for $\tilde{A}$.
    \label{lem:diagonalisableRankOnePerturbation}
\end{lemma}
\begin{proof}
$A$ of the form $Q D Q^{-1}$ for an invertible $Q$ and a diagonal matrix $D$. Assume that the first diagonal entry of $D$ is $1$. Columns of $Q$ form an eigenbasis and we may impose that the first column is $u$. Rows of $Q^{-1}$ form an eigenbasis of $A^T$, set $v_0$ the vector corresponding to the first row. Since $1$ is a simple eigenvalue, the corresponding eigenspace has dimension $1$ and there exists $\alpha \neq 0$ such that $v = \alpha v_0$. We have $u^Tv = 1$ by assumption and $u^T v_0 = 1$ because $Q^{-1}Q = I$, this shows that $\alpha = 1$ and therefore $v$ is the first row of $Q^{-1}$.

    We have $v^Tu = 1$ and therefore $\tilde{A} u = A u - u = 0$. Let $\tilde{u}$ be a different column of $Q$ corresponding to an eigenvector of $A$ associated to eigenvalue $d$, we have $v^T \tilde{u} = 0$ so that $\tilde{A} \tilde{u} = A \tilde{u} = d\tilde{u}$. This concludes the proof.
\end{proof}

\begin{lemma}[Uniform convergence leads to continuous differentiable limit]\label{lem:cont-diff}
    Let $U \subset \RR^p$ be open and $(f_k)_{k \in \NN}$ be a sequence of continuously differentiable functions from $U$ to $\RR$ converging pointwise to $\bar{f} \colon U \to \RR$, such that $\nabla f_k$ converges pointwise, locally uniformly on $U$. Then $\bar{f}$ is continuously differentiable on $U$ and $\nabla \bar{f} = \lim_{k \to \infty} \nabla f_k$.
    \label{lem:gradientUniformConvergence}
\end{lemma}
\begin{proof}
    Let $g = \lim_{k \to \infty} \nabla f_k$ the pointwise limit. By local uniform convergence, $g$ is continuous on $U$. Fix any $x \in U$ and any $v \in \RR^n$ and set $I$ a closed interval such that $x + tv \in U$ for all $t \in I$ and $0$ is in the interior of $I$ (such interval exists because $U$ is open). The sequence of univariate functions $h_k \colon t \mapsto f_k(x + t v)$ is continuously differentiable and satisfy for all $k$ and all $t \in I$
    \begin{align*}
        h_k'(t) = \left\langle \nabla f_k(x + tv), v\right\rangle .
    \end{align*}
    The derivatives $h_k'$ converge uniformly on $I$ to $\left\langle g(x + tv), v\right\rangle$ which is continuous in $t$. Therefore the function $\bar{h} \colon t \mapsto \bar{f}(x + tv)$ is continuously differentiable with derivative given by $\left\langle g(x + tv), v\right\rangle$, by uniform convergence of derivatives. Since $x\in U$ and $v \in \RR^n$ were arbitrary, this implies that $\bar{f}$ admits continuous partial derivatives and it is therefore continuously differentiable with gradient $g$.
\end{proof}

\begin{lemma}[Centering]\label{lem:variationPseudonorm}
    For $x,x' \in \RR^n$, 
    \begin{align*}
        \left\|\cent(x) - \cent(x')\right\|_\infty \leq \|x - x'\|_\var,
    \end{align*}
    where $\cent$ is defined in~\eqref{eq:lcenter} and $\|\cdot\|_\var$ is defined in~\eqref{eq:fvar}.
\end{lemma}
\begin{proof}
    Note that for $f \in \RR^n$ and $a \in \RR$, $\|f+ a 1_n\|_\var = \|f\|_\var$.
    Set $f = \cent(x) - \cent(x')$, we have $1_n^Tf = \sum_{i=1}^n f_i = 0$ so that 
    \begin{align*}
        \min_i f_i \leq \sum_{i=1}^n f_i = 0 \leq \max_i f_i.
    \end{align*}
    This implies the following
    \begin{align*}
        &\quad\|f\|_\infty = \max_{i} |f_i| \\
        &= \max_{i} \max\{f_i, -f_i\} \\
        &= \max\{ \max_i f_i, \max_i -f_i\} \\
        &= \max\{ \max_i f_i, -\min_i f_i\}\\
        &\leq \max\{ \max_i f_i - \min_if_i, \max_i f_i -\min_i f_i\} \\
        &=\|f\|_\var.
    \end{align*}
    Now $f = \cent(x) - \cent(x') = x - x' +1_n \left(\frac{1}{n} \sum_{i=1}^n x'_i - \frac{1}{n} \sum_{i=1}^n x_i\right)$, so that $\|f\|_\var = \|x - x'\|_\var$ which concludes the proof.
\end{proof}

\begin{lemma}
    Let $\rho \in (0,1)$, $c > 0$ and $(\delta_k)_{k \in \NN}$ be a positive sequence such that for all $k \in \NN$,
    \begin{align}\label{eq:tightestLemmaHyp}
        \delta_{k+1} \leq (\rho + c \rho^{k+1}) \delta_k + c \rho^{k+1}.
    \end{align}
    Then, for all $k \in \NN$, such that $k \geq \frac{\rho}{1-\rho}$, we have
    \begin{align*}
        \delta_k \leq \rho^k \exp\left(1 + \frac{c}{1-\rho} \right) \left(\delta_0  +  c(k+1)\right) .
    \end{align*}
    \label{lem:realSequence3}
\end{lemma}
\begin{proof}
    Fix $\alpha \in (0,1)$ to be chosen latter.
    Dividing~\eqref{eq:tightestLemmaHyp} on both sides by $c\rho^{(k+1)/\alpha}$, we have for all $k \in \NN$,
    \begin{align*}
        \frac{\delta_{k+1}}{c\rho^{\alpha(k+1)}} &\leq \frac{\delta_{k}}{c\rho^{\alpha k}} \left(\frac{\rho c\rho^{ \alpha k}}{c\rho^{\alpha (k+1)}} + \frac{c \rho^{k+1}c\rho^{\alpha k}}{c\rho^{\alpha (k+1)}} \right) + \frac{ c \rho^{k+1}}{c\rho^{\alpha(k+1)}}\\
        &=\frac{\delta_{k}}{c\rho^{\alpha k}} \left(\rho^{1 - \alpha} +c\rho^{k+1 - \alpha} \right) + \rho^{(k+1)(1-\alpha)}\\
        &\leq  \frac{\delta_k}{c\rho^{\alpha k}} \left(1 + c\rho^{k+1 - \alpha} \right) + \rho^{(k+1)(1-\alpha)} .
    \end{align*}
    Setting for all $k \in \NN$,
    \begin{align*}
        z_k = \frac{\delta_k}{c\rho^{\alpha k}}, \quad
        \alpha_k = c\rho^{k+1 - \alpha}, \quad \text{and} \quad
        \beta_k = \rho^{(k+1)(1-\alpha)},
    \end{align*}
    we apply Lemma \ref{lem:realSequence1} to obtain the result.
    As $(\alpha_{k})_{k \in \NN}$ and $(\beta_{k})_{k \in \NN}$ are geometric sequences, we have $\sum_{i=0}^{+\infty} \alpha_i = \frac{c\rho^{1-\alpha}}{1-\rho}\leq \frac{c}{1-\rho}$ and $\sum_{i=0}^{+\infty} \beta_i = \frac{\rho^{1-\alpha}}{1-\rho^{1-\alpha}} \leq \frac{1}{1 - \rho^{1-\alpha}}$ , so that for all $k \in \NN$,
    \begin{align*}
        \frac{\delta_k}{c\rho^{\alpha k}}  &=z_k \\
        & \leq \exp\left( \sum_{i=0}^{+\infty} \alpha_i\right)\left(z_0 +  \sum_{j=0}^{+\infty} \beta_j\right) \\
        &= \exp\left( \frac{c}{1-\rho} \right)\left(\frac{\delta_0}{c}  +  \frac{1}{1-\rho^{1-\alpha}}\right) .
    \end{align*}
    Since $\alpha$ was arbitrary, the preceding holds for all $k \in \NN$ and $\alpha \in (0,1)$. Fix $k \in \NN$ such that $k > \frac{\rho}{1-\rho}$. Setting $\alpha = 1 + \log\left(1 + \frac{1}{k}\right) / \log(\rho)$, since $\rho \in (0,1)$, we have
    \begin{align*}
       0 = 1 + \log\left(1 + \frac{1 - \rho}{\rho}\right) / \log(\rho) < \alpha < 1.
    \end{align*}
    We have 
    \begin{align*}
        \rho^{\alpha k} &= \rho^k \rho^{k\log((k+1)/k) / \log(\rho)} =\rho^k\left(1 + \frac{1}{k}\right)^k \leq e \rho^k ,
    \end{align*}
    and
    \begin{align*}
        \frac{1}{1-\rho^{1-\alpha}} & = \frac{1}{1-\rho^{-\log(1 + 1/k) / \log(\rho)}} = \frac{1}{1-\rho^{\log(k / (k+1)) / \log(\rho)}} \\
        &=\frac{1}{1-\frac{k}{k+1}} =  k+1.
    \end{align*}
    Therefore, for all $k \geq \frac{\rho}{1-\rho}$,
    \begin{align*}
        \delta_k \leq \rho^k \exp\left(1 + \frac{c}{1-\rho} \right) \left(\delta_0  +  c(k+1)\right) ,
    \end{align*}
    proving our claim.
\end{proof}

\section*{Acknowledgments}
E.P. and S.V. would like to thank Jérôme Bolte for fruitful and inspiring discussions.
S.V. thanks Jeremy Cohen and Titouan Vayer for raising the issue of the convergence of the derivatives of the Sinkhorn--Knopp iterates during a seminar in Lyon, France.

\bibliographystyle{siamplain}
\bibliography{references}

\begin{thebibliography}{10}

\bibitem{adams2011ranking}
{\sc R.~P. Adams and R.~S. Zemel}, {\em Ranking via sinkhorn propagation},
  2011, \url{https://arxiv.org/abs/1106.1925}.

\bibitem{agueh2011barycenters}
{\sc M.~Agueh and G.~Carlier}, {\em Barycenters in the wasserstein space}, SIAM
  Journal on Mathematical Analysis, 43 (2011), pp.~904--924.

\bibitem{ben2003generalized}
{\sc A.~Ben-Israel and T.~N. Greville}, {\em Generalized inverses: theory and
  applications}, vol.~15, Springer Science \& Business Media, 2003.

\bibitem{birkhoff1957extensions}
{\sc G.~Birkhoff}, {\em Extensions of jentzsch's theorem}, Transactions of the
  American Mathematical Society, 85 (1957), pp.~219--227.

\bibitem{bolte2022automatic}
{\sc J.~Bolte, E.~Pauwels, and S.~Vaiter}, {\em Automatic differentiation of
  nonsmooth iterative algorithms}, arXiv preprint arXiv:2206.00457,  (2022).

\bibitem{bonneel2016wasserstein}
{\sc N.~Bonneel, G.~Peyr\'{e}, and M.~Cuturi}, {\em Wasserstein barycentric
  coordinates: Histogram regression using optimal transport}, ACM Transactions
  on Graphics, 35 (2016), \url{https://doi.org/10.1145/2897824.2925918}.

\bibitem{jax2018github}
{\sc J.~Bradbury, R.~Frostig, P.~Hawkins, M.~J. Johnson, C.~Leary,
  D.~Maclaurin, G.~Necula, A.~Paszke, J.~Vander{P}las, S.~Wanderman-{M}ilne,
  and Q.~Zhang}, {\em {JAX}: composable transformations of {P}ython+{N}um{P}y
  programs}, 2018, \url{http://github.com/google/jax}.

\bibitem{caron2020unsupervised}
{\sc M.~Caron, I.~Misra, J.~Mairal, P.~Goyal, P.~Bojanowski, and A.~Joulin},
  {\em Unsupervised learning of visual features by contrasting cluster
  assignments}, in NeurIPS, vol.~33, 2020, pp.~9912--9924.

\bibitem{christianson1994reverse}
{\sc B.~Christianson}, {\em Reverse accumulation and attractive fixed points},
  Optimization Methods and Software, 3 (1994), pp.~311--326.

\bibitem{cuturi2013sinkhorn}
{\sc M.~Cuturi}, {\em Sinkhorn distances: Lightspeed computation of optimal
  transport}, in NeurIPS, vol.~26, 2013.

\bibitem{cuturi2022optimal}
{\sc M.~Cuturi, L.~Meng-Papaxanthos, Y.~Tian, C.~Bunne, G.~Davis, and
  O.~Teboul}, {\em Optimal transport tools (ott): A jax toolbox for all things
  wasserstein}, arXiv preprint arXiv:2201.12324,  (2022).

\bibitem{cuturi2020supervised}
{\sc M.~Cuturi, O.~Teboul, J.~Niles-Weed, and J.-P. Vert}, {\em Supervised
  quantile normalization for low-rank matrix approximation}, in ICML, 2020.

\bibitem{cuturi2019differentiable}
{\sc M.~Cuturi, O.~Teboul, and J.-P. Vert}, {\em Differentiable ranking and
  sorting using optimal transport}, in NeurIPS, vol.~32, 2019.

\bibitem{eisenberger2022unified}
{\sc M.~Eisenberger, A.~Toker, L.~Leal-Taix{\'e}, F.~Bernard, and D.~Cremers},
  {\em A unified framework for implicit sinkhorn differentiation}, in CVPR,
  2022.

\bibitem{flamary2021pot}
{\sc R.~Flamary, N.~Courty, A.~Gramfort, M.~Z. Alaya, A.~Boisbunon, S.~Chambon,
  L.~Chapel, A.~Corenflos, K.~Fatras, N.~Fournier, L.~Gautheron, N.~T. Gayraud,
  H.~Janati, A.~Rakotomamonjy, I.~Redko, A.~Rolet, A.~Schutz, V.~Seguy, D.~J.
  Sutherland, R.~Tavenard, A.~Tong, and T.~Vayer}, {\em Pot: Python optimal
  transport}, Journal of Machine Learning Research, 22 (2021), pp.~1--8.

\bibitem{franklin1989scaling}
{\sc J.~Franklin and J.~Lorenz}, {\em On the scaling of multidimensional
  matrices}, Linear Algebra and its Applications, 114-115 (1989), pp.~717--735,
  \url{https://doi.org/https://doi.org/10.1016/0024-3795(89)90490-4}.

\bibitem{genevay2018learning}
{\sc A.~Genevay, G.~Peyr\'{e}, and M.~Cuturi}, {\em Learning generative models
  with sinkhorn divergences}, in AISTAT, vol.~84, 2018, pp.~1608--1617.

\bibitem{gilbert1992automatic}
{\sc J.~C. Gilbert}, {\em Automatic differentiation and iterative processes},
  Optimization Methods and Software, 1 (1992), pp.~13--21,
  \url{https://doi.org/10.1080/10556789208805503}.

\bibitem{griewank1993derivative}
{\sc A.~Griewank, C.~Bischof, G.~Corliss, A.~Carle, and K.~Williamson}, {\em
  Derivative convergence for iterative equation solvers}, Optimization Methods
  and Software, 2 (1993), pp.~321--355,
  \url{https://doi.org/10.1080/10556789308805549}.

\bibitem{hashimoto2016learning}
{\sc T.~Hashimoto, D.~Gifford, and T.~Jaakkola}, {\em Learning population-level
  diffusions with generative rnns}, in ICML, vol.~48, 2016, pp.~2417--2426.

\bibitem{horn2013matrix}
{\sc R.~A. Horn and C.~R. Johnson}, {\em Matrix analysis}, Cambridge University
  Press, Cambridge, second~ed., 2013.

\bibitem{linnainmaa1976taylor}
{\sc S.~Linnainmaa}, {\em Taylor expansion of the accumulated rounding error},
  BIT Numerical Mathematics, 16 (1976), pp.~146--160.

\bibitem{lorraine2020optimizing}
{\sc J.~Lorraine, P.~Vicol, and D.~Duvenaud}, {\em Optimizing millions of
  hyperparameters by implicit differentiation}, in AISTAT, vol.~108, 2020,
  pp.~1540--1552.

\bibitem{luise2019differential}
{\sc G.~Luise, A.~Rudi, M.~Pontil, and C.~Ciliberto}, {\em Differential
  properties of sinkhorn approximation for learning with wasserstein distance},
  in NeurIPS, vol.~31, 2018.

\bibitem{mehmood2020automatic}
{\sc S.~Mehmood and P.~Ochs}, {\em Automatic differentiation of some
  first-order methods in parametric optimization}, in AISTAT, 2020,
  pp.~1584--1594.

\bibitem{peyre2019book}
{\sc G.~Peyr{\'e} and M.~Cuturi}, {\em Computational optimal transport},
  Foundations and Trends in Machine Learning, 51 (2019), pp.~1--44,
  \url{https://doi.org/10.1561/2200000073}.

\bibitem{polyak1987introduction}
{\sc B.~Polyak}, {\em Introduction to optimization}, in Optimization Software,
  Publications Division, Citeseer, 1987.

\bibitem{ruschendorf1995convergence}
{\sc L.~Ruschendorf}, {\em {Convergence of the Iterative Proportional Fitting
  Procedure}}, The Annals of Statistics, 23 (1995), pp.~1160 -- 1174,
  \url{https://doi.org/10.1214/aos/1176324703}.

\bibitem{scroggs1966alternate}
{\sc J.~E. Scroggs and P.~L. Odell}, {\em An alternate definition of a
  pseudoinverse of a matrix}, SIAM Journal on Applied Mathematics, 14 (1966),
  pp.~796--810.

\bibitem{sinkhorn1964relationship}
{\sc R.~Sinkhorn}, {\em {A Relationship Between Arbitrary Positive Matrices and
  Doubly Stochastic Matrices}}, The Annals of Mathematical Statistics, 35
  (1964), pp.~876 -- 879.

\bibitem{sinkhorn1967diagonal}
{\sc R.~Sinkhorn}, {\em Diagonal equivalence to matrices with prescribed row
  and column sums}, The American Mathematical Monthly, 74 (1967), pp.~402--405.

\bibitem{thornton2022rethinking}
{\sc J.~Thornton and M.~Cuturi}, {\em Rethinking initialization of the sinkhorn
  algorithm}, 2022, \url{https://doi.org/10.48550/ARXIV.2206.07630}.

\bibitem{wengert1964simple}
{\sc R.~E. Wengert}, {\em A simple automatic derivative evaluation program},
  Commun. ACM, 7 (1964), p.~463–464,
  \url{https://doi.org/10.1145/355586.364791}.

\end{thebibliography}

\end{document}